\documentclass[11pt,letterpaper,reqno]{amsart}

\usepackage[scale=0.85,centering,foot=1cm]{geometry}
\usepackage{graphicx}
\usepackage{amsmath,amssymb,amsthm,bm}
\usepackage{mathtools}
\usepackage{epsfig,color,cases,url}
\usepackage[noadjust]{cite}
\usepackage[foot]{amsaddr}
\usepackage{enumerate}

\newcommand{\Integer}{\mathbb{Z}}
\newcommand{\Natural}{\mathbb{N}}
\newcommand{\Naturalstar}{\mathbb{N}_*}
\newcommand{\Real}{\mathbb{R}}

\newcommand{\hilbert}{\mathcal{H}}

\newcommand{\innprod}[2]{\left\langle{#1},{#2}\right\rangle}
  
\newcommand{\norm}[1]{\left\|{#1}\right\|}

\DeclareMathOperator{\argmin}{arg\,min}

\DeclareMathOperator{\relinterior}{ri}
\DeclareMathOperator{\interior}{int}

\DeclareMathOperator{\conv}{conv}
\DeclareMathOperator{\sign}{sgn}

\DeclareMathOperator{\lev}{lev_{\leq 0}}

\DeclareMathOperator{\Fix}{Fix}

\DeclareMathOperator{\supp}{supp}
\DeclareMathOperator{\graph}{gph}

\theoremstyle{definition}
\newtheorem{theorem}{Theorem}
\newtheorem{algo}[theorem]{Algorithm}

\newtheorem{fact}[theorem]{Fact}

\newtheorem{lemma}[theorem]{Lemma}

\newtheorem{definition}[theorem]{Definition}
\newtheorem{example}[theorem]{Example}
\newtheorem{assumption}[theorem]{Assumption}
\newtheorem{problem}{Problem}

\begin{document}

\title{The Adaptive Projected Subgradient Method\\ constrained by
  families of quasi-nonexpansive mappings\\ and its application to
  online learning}

\author{Konstantinos Slavakis$^1$}
\address{$^1$University of Peloponnese, Department of
  Telecommunications Science and Technology, Karaiskaki St., Tripolis
  22100, Greece. Tel: +30.2710.37.2204, Fax: +30.2710.37.2160.}
\email{slavakis@uop.gr}

\author{Isao Yamada$^2$}
\address{$^2$Tokyo Institute of Technology, Department of
  Communications and Integrated Systems, S3-60, Tokyo 152-8550,
  Japan. Tel: +81.3.5734.2503, Fax: +81.3.5734.2905.}
\email{isao@sp.ss.titech.ac.jp}

\begin{abstract}
Many online, i.e., time-adaptive, inverse problems in signal
processing and machine learning fall under the wide umbrella of the
asymptotic minimization of a sequence of non-negative, convex, and
continuous functions. To incorporate a-priori knowledge into the
design, the asymptotic minimization task is usually constrained on a
fixed closed convex set, which is dictated by the available a-priori
information. To increase versatility towards the usage of the
available information, the present manuscript extends the Adaptive
Projected Subgradient Method (APSM) by introducing an algorithmic
scheme which incorporates a-priori knowledge in the design via a
sequence of strongly attracting quasi-nonexpansive mappings in a real
Hilbert space. In such a way, the benefits offered to online learning
tasks by the proposed method unfold in two ways: 1) the rich class of
quasi-nonexpansive mappings provides a plethora of ways to cast
a-priori knowledge, and 2) by introducing a sequence of such mappings,
the proposed scheme is able to capture the time-varying nature of
a-priori information. The convergence properties of the algorithm are
studied, several special cases of the method with wide applicability
are shown, and the potential of the proposed scheme is demonstrated by
considering an increasingly important, nowadays, online sparse
system/signal recovery task.
\end{abstract}

\keywords{Projection, quasi-nonexpansive mapping, fixed point,
  subgradient, asymptotic minimization, online learning, system
  identification, sparsity}

\subjclass[2010]{47N10, 47H09, 94A12, 65J22, 65K10, 90C25}
\date{\today}

\maketitle

\section{Introduction}\label{sec:intro}

Many online, i.e., time-adaptive, inverse problems in signal
processing and machine learning can be recast as follows \cite{Haykin,
  sayed.book, liu.principe.haykin, KivinenEtal, EngelEtalKRLS,
  ChenHero09, angelosante.occd, myyy.soft.thres.icassp10,
  Slav.Kops.Theo.ICASSP10, kst@tsp11, YKostasY, KostasClassifyTSP,
  ksi.Beam.TSP, tsy.sp.magazine, sbt.mimo.tnn}; if the non-negative
integer $n\in\Natural$ denotes discrete time, having at our disposal a
sequence of multidimensional data $(\bm{a}_n,d_n)_{n\in\Natural}
\subset\Real^L \times\Real$, the objective of an online learning
method is to infer a possibly time-varying unknown mapping $x_*:
\Real^L \rightarrow\Real$, which relates the previous data under the
following model:
\begin{equation}
d_n = x_*(\bm{a}_n) + \zeta_n,\quad \forall n\in\Natural. \label{model}
\end{equation}
In other words, at the $n$-th time instant, the $L$-dimensional input
signal $\bm{a}_n$ interacts with the signal/system which underlies
$x_*$, and our observation is the real valued $d_n$ which is
contaminated by the additive noise $\zeta_n$.

Online learning methods show distinct differences from their
\textit{batch} counterparts due to the following fundamental reason:
batch optimization methods are mobilized \textit{after} all the
necessary data are available to the designer, whereas, in the online
scenario, the sequential nature of the data
$(\bm{a}_n,d_n)_{n\in\Natural}$ dictates that at each time instant
$n$, the newly arriving $(\bm{a}_n,d_n)$ should be efficiently
incorporated into the learning process, without the need of solving
the optimization task from scratch. Such a sequential mode is not
prescribed only by the need for computational efficiency and
savings. The online processing of data becomes an efficient tool also
in cases of dynamic scenarios, where not only the probability density
function of the input data $(\bm{a}_n)_{n\in\Natural}$ changes with
time, but also where the unknown mapping $x_*$ shows a time-varying
nature. In such time dependent environments, and in order to monitor
the time variations of the underlying signals and systems, the
designer is compelled to gradually disregard data which are associated
to the remote past, and to put emphasis on recently received
$(\bm{a}_n,d_n)$. It becomes clear that flexible and multifaceted
online learning tools are needed in order to deal with fast emerging
signal processing and machine learning applications, like
sparsity-aware learning \cite{ChenHero09, angelosante.occd,
  myyy.soft.thres.icassp10, kst@tsp11}, time-adaptive sensor networks
\cite{simos.ieee.tsp, cattivelli.birds}, etc.

The unknown mapping $x_*$ of \eqref{model} could be either linear or
non-linear. Our assumption on the linearity or not of $x_*$ dictates the
choice of possible spaces into which we perform our search for $x_*$. If
$x_*$ is assumed linear, then our working space becomes the classical
Euclidean $\Real^L$ \cite{Haykin, sayed.book}. On the other hand, if
$x_*$ is assumed non-linear, a mathematical sound way to model a fairly
large amount of non-linear systems is to work in a possibly infinite
dimensional \textit{Reproducing Kernel Hilbert Space (RKHS)}
\cite{Aronszajn}; a strategy which has been particularly successful in
machine learning and pattern recognition tasks \cite{ScholkopfSmola,
theo.koutrou, liu.principe.haykin, EngelEtalKRLS,
KostasClassifyTSP, ksi.Beam.TSP, tsy.sp.magazine,
Bouboulis.kernel.image}. Since the Euclidean $\Real^L$ is a renowned
Hilbert space, and in order to offer a unifying framework for linear and
non-linear systems, the stage of the following discussion will be based
on a real Hilbert space $\hilbert$.

Given an estimate $x\in\hilbert$ of the unknown $x_*$, the most common
way to validate $x$, with respect to the model \eqref{model}, is to
penalize the disagreement of the observed output $d_n$ with
$x(\bm{a}_n)$, i.e., the real-valued difference $x(\bm{a}_n) -d_n$. A
classical way to quantify such a perception of loss is to use the
quadratic function in order to form the penalty $\left(x(\bm{a}_n)
-d_n \right)^2$. The popularity of the quadratic loss function is
based on its optimality in estimation tasks where the contaminating
noise process $(\zeta_n)_{n\in\Natural}$ is Gaussian
\cite{stanford.stats.book}. However, in order to establish a general
framework for estimation problems, where the noise process is not
constrained to be Gaussian, and in order to build estimators which
show robustness to a wide variety of outliers, we give ourselves the
freedom to employ any convex function $\mathcal{L}: \Real\rightarrow
[0,\infty)$, and not just the quadratic one, in order to quantify our
  perception of loss (see for example \cite{sbt.mimo.tnn}). Having the
  data $(\bm{a}_n, d_n)$ as parameters in the design, the following
  function is naturally defined on the space $\hilbert$ of our
  estimates: $\Theta_n: \hilbert \rightarrow [0,\infty): x\mapsto
    \mathcal{L}(x(\bm{a}_n) -d_n)$. Due to the online nature of the
    problem, i.e., the sequential data $(\bm{a}_n,
    d_n)_{n\in\Natural}$, we end up in a sequence of loss functions
    $(\Theta_n)_{n\in\Natural}$. We stress here that since
    $\mathcal{L}$ can be any convex function, $\Theta_n$ is not bound
    to be differentiable.

Theory, e.g., Bayesian inference \cite{stanford.stats.book}, as well
as everyday practice suggest that apart from the information included
in the training sequence $(\bm{a}_n, d_n)_{n\in\Natural}$, estimation
is enhanced if one employs also the \textit{a-priori knowledge} about
the unknown system $x_*$. We will abide here by the set theoretic
estimation approach \cite{Combettes.foundations} and quantify the
a-priori knowledge as a closed convex set $C$ in $\hilbert$. The first
attempt to attack the task of online learning as the asymptotic
minimization of a sequence $(\Theta_n)_{n\in\Natural}$, over a
nonempty closed convex set $C$, was given in \cite{YamadaAPSM,
  YamadaOguraAPSMNFAO}, by means of the following simple iteration,
called the \textit{Adaptive Projected Subgradient Method (APSM);} for
an arbitrary initial point $u_0\in\hilbert$, let
\begin{equation*}
\forall n\in\Natural, \quad u_{n+1} := \begin{cases}
P_C\left(u_n - \lambda_n
\frac{\Theta_n(u_n)}{\norm{\Theta_n'(u_n)}^2}
\Theta_n'(u_n) \right), & \text{if}\ \Theta_n'(u_n) \neq
0,\\
P_C(u_n), & \text{if}\ \Theta_n'(u_n) = 0,
\end{cases}
\end{equation*}
where $\lambda_n\in(0,2)$, $P_C$ stands for the metric projection
mapping onto $C$, and $\Theta'_n(u_n)$ denotes any subgradient of
$\Theta_n$ at $u_n$, $\forall n\in\Natural$. The previous recursion is
a time-adaptive generalization of the classical algorithm of Polyak
\cite{PolyakUnsmooth}, which deals with the minimization problem of a
\textit{fixed}, non-smooth, convex and continuous function $\Theta$
over $C$. Besides the new directions for online learning
\cite{tsy.sp.magazine}, the previous recursion has offered
also a unification of several standard algorithms in classical
adaptive filtering \cite{Haykin, sayed.book}. Indeed, by letting
$\hilbert:= \Real^L$, for an appropriately chosen sequence
$(\Theta_n)_{n\in\Natural}$, and by substituting $P_C$ with the
identity mapping, the previous recursion \cite{YamadaAPSM,
  YamadaOguraAPSMNFAO, YKostasY} results in the classical Normalized
Least Mean Squares (NLMS) \cite{NagumoNoda, AlbertGardner} and the,
vastly used nowadays, Affine Projection Algorithm (APA)
\cite{HinamotoMaekawa, OzekiUmeda}.

It is often the case that a single closed convex set $C$, or even
better, a single metric projection mapping $P_C$, cannot capture the
diversity of the a-priori knowledge in signal processing
applications. For example, in a robust beamforming problem
\cite{sy.wideband}, the a-priori knowledge is usually expressed as $C=
\bigcap_{m=1}^M C_m$, where $\{C_m\}_{m=1}^M$ is a number of closed
convex sets, with associated projection mappings $\{P_{C_m}\}_{m=1}^M$
that are usually easy to compute. However, an analytic expression for
$P_C$ might not be available \cite{sy.wideband}. Secondly, erroneous
a-priori information may result into an empty $C= \bigcap_{m=1}^M C_m=
\emptyset$ \cite{yukawa.splitting, sy.wideband}. How is it possible to
deal with multiple closed convex sets $\{C_m\}_{m=1}^M$ where an
analytical expression of $P_C$ is not available, or the
$\{C_m\}_{m=1}^M$ share an empty intersection? Avoiding the
straightforward and recently popular solution of relaxing the original
constraints, the study in \cite{Kostas.APSM.NFAO} provides with a
solution to the previous problem and extends \cite{YamadaAPSM,
  YamadaOguraAPSMNFAO} by using a mapping $T$, in the place of $P_C$,
which belongs to the general class of strongly attracting nonexpansive
mappings. Indeed, the method \cite{Kostas.APSM.NFAO} demonstrated its
potential in a wide variety of online learning tasks, which span from
classical linear adaptive filtering \cite{yukawa.splitting} to
non-linear classification and regression tasks \cite{tsy.sp.magazine}.

It is natural to ask now whether we can add more freedom to the usage
of the a-priori knowledge. Our motivation is based on a couple of
elementary observations. First, given the well-known fact that a
nonempty closed convex set $C$ is the set of all minimizers of the
distance function $d(\cdot,C)$ to $C$, one of the ways to visualize
a-priori knowledge could be the set of all minimizers of a generally
non-smooth convex function defined on an appropriate Hilbert space
$\hilbert$. Secondly, it is often the case in practice where a
minimizer of a convex function cannot be reached either by an
analytical formula or a computationally cheap process. A powerful
mapping, whose recursive application is known to minimize a generally
non-differentiable convex function, is the subgradient projection
mapping \cite{BauschkeCombettesWeak2Strong, yo.hybrid.quasi.ne,
  bauschke.combettes.book}. It is also known that this operator
belongs to the class of quasi-nonexpansive mappings
\cite{BauschkeCombettesWeak2Strong, yo.hybrid.quasi.ne,
  bauschke.combettes.book}, which strictly contains all the strongly
attracting nonexpansive mappings, utilized in
\cite{Kostas.APSM.NFAO}. Now, the question arises naturally: does the
APSM still operate when constrained by the general class of
quasi-nonexpansive mappings, and can we, thus, devise a method with
more freedom in incorporating a-priori information, than in the
studies of \cite{YamadaAPSM, YamadaOguraAPSMNFAO, Kostas.APSM.NFAO}?
Given the wide applicability of the APSM in online learning tasks
\cite{tsy.sp.magazine}, it is anticipated that such a generalization
will add further flexibility to the APSM in order to tackle more
challenging online learning tasks, which have been recently emerging
both in signal processing and machine learning \cite{ChenHero09,
  angelosante.occd, myyy.soft.thres.icassp10, kst@tsp11,
  simos.ieee.tsp, cattivelli.birds}.

The present manuscript introduces an extension of the APSM
\cite{Kostas.APSM.NFAO, YamadaAPSM, YamadaOguraAPSMNFAO}, towards a
more flexible usage of the a-priori information, in two ways: 1) by
considering a strictly larger class of mappings than in
\cite{Kostas.APSM.NFAO, YamadaAPSM, YamadaOguraAPSMNFAO}, and in
particular, operators taken from the rich family of
\textit{quasi-nonexpansive mappings}, and 2) by letting these mapping
to be \textit{time-varying} in order to capture the, quite often in
signal processing and machine learning applications, dynamic nature of
the a-priori information. Put in mathematical terms, the problem to be
studied is the following.

\begin{problem}[Constrained asymptotic minimization task]\label{the.problem}
Given a sequence of convex, continuous, and not necessarily
differentiable functions $(\Theta_n: \hilbert\rightarrow
[0,\infty))_{n\in\Natural}$, and a sequence of strongly attracting
  quasi-nonexpansive mappings $(T_n: \hilbert \rightarrow
  \hilbert)_{n\in\Natural}$, with nonempty fixed point sets
  $(\Fix(T_n))_{n\in\Natural}$, we are looking for a sequence
  $(u_n)_{n\in\Natural}$ that asymptotically minimizes
  $(\Theta_n)_{n\in\Natural}$ over
  $(\Fix(T_n))_{n\in\Natural}$. Strictly speaking, our objective is to
  generate a $(u_n)_{n\in\Natural}$ such that
  $\lim_{n\rightarrow\infty}\Theta_n(u_n)=0$, and the set of its
  strong cluster points $\mathfrak{S}((u_n)_{n\in\Natural})$ lies in
  $\limsup_{n\rightarrow\infty} \Fix(T_n)$, i.e.,
  $\mathfrak{S}((u_n)_{n\in\Natural})\subset
  \limsup_{n\rightarrow\infty} \Fix(T_n)$.
\end{problem}

Our algorithmic tool to tackle the previous optimization task is the
following.

\begin{algo}\label{the.algo}
Given an arbitrary initial point $u_0\in\hilbert$, generate the
following sequence:
\begin{equation}
\forall n\in\Natural, \quad u_{n+1} := \begin{cases}
T_n\left(u_n - \lambda_n
\frac{\Theta_n(u_n)}{\norm{\Theta_n'(u_n)}^2}
\Theta_n'(u_n) \right), & \text{if}\ \Theta_n'(u_n) \neq
0,\\
T_n(u_n), & \text{if}\ \Theta_n'(u_n) = 0,
\end{cases}\label{basic.recursion}
\end{equation}
where $\lambda_n\in(0,2)$ and
$\Theta_n'(u_n)$ stands for any subgradient of $\Theta_n$ at $u_n$,
$\forall n\in\Natural$.
\end{algo}

The manuscript is organized as follows. A series of necessary
definitions and facts are included in Section~\ref{sec:junbi}. The
algorithm and its convergence analysis follow in
Section~\ref{sec:algo}. Special cases of the algorithm, with a wide
application range in online learning, can be found in
Section~\ref{sec:examples}. The potential of the method is shown in
Section~\ref{sec:application} by introducing a low-complexity
time-adaptive learning technique for the increasingly important,
nowadays, sparse system/signal recovery task.

\section{Preliminaries}\label{sec:junbi}

We start with several notations which will be frequently used in the
sequel.

The set of all non-negative integers, positive integers, and real
numbers will be denoted by $\Natural$, $\Naturalstar$, and $\Real$,
respectively. The set of all subsequences of $\Natural$ will be
denoted by $\Natural_{\infty}^{\#}$, i.e., $\Natural_{\infty}^{\#}:=
\{N\subset \Natural: N\ \text{is infinite}\}$
\cite{Rockafellar.Wets}. Any $N\in \Natural_{\infty}^{\#}$ can be also
denoted by the standard way of $N=(n_k)_{k\in \Natural}$. Define,
also, $\Natural_{\infty}:= \{N\subset \Natural: \Natural\setminus
N\ \text{is finite}\}$ \cite{Rockafellar.Wets}. In other words,
$\Natural_{\infty}$ contains all the ``neighborhoods of $\infty$'',
with respect to $\Natural$, while $\Natural_{\infty}^{\#}$ is its
associated ``grill'' \cite{Rockafellar.Wets}.

Henceforth, the symbol $\hilbert$ will stand for a real Hilbert space,
equipped with an inner product $\innprod{\cdot}{\cdot}$, and a norm
$\norm{\cdot}:=\sqrt{\innprod{\cdot}{\cdot}}$. In the case where
$\hilbert$ becomes the Euclidean $\Real^L$, $L\in \Naturalstar$, any
element of $\Real^L$ will be denoted by boldfaced symbols. The inner
product of $\Real^L$ will be the classical vector dot product, i.e.,
$\innprod{\bm{v}_1}{\bm{v}_2}:= \bm{v}_1^t \bm{v}_2$, $\forall
\bm{v}_1, \bm{v}_2\in \Real^L$, where the superscript $t$ stands for
vector/matrix transposition.

Given an $x\in\hilbert$ and a $\rho>0$, an open ball is defined as the
set $B(x,\rho):= \{v\in\hilbert:\ \norm{x-v}<\rho\}$, while a closed
ball $B[x,\rho]:= \{v\in\hilbert:\ \norm{x-v}\leq \rho\}$. Given $S,
\Upsilon\subset \hilbert$, the \textit{relative interior of $S$ with
  respect to $\Upsilon$} is defined as $\relinterior_{\Upsilon} S:=
\{\mathring{v}\in S:\ \exists \rho>0, \emptyset\neq
(B(\mathring{v},\rho) \cap \Upsilon) \subset S\}$. The
\textit{interior} of $S$ is defined as $\interior S:=
\relinterior_{\hilbert} S$.

Given $S\subset\hilbert$, define the \textit{distance function to $S$}
as follows: $d(\cdot,S): \hilbert\rightarrow [0,\infty): x \mapsto
  d(x,S):= \inf\{\norm{x-v}: v\in S\}$. Given any nonempty
  \textit{closed convex} set $C\subset\hilbert$, the \textit{(metric)
    projection onto $C$} is defined as the mapping $P_C: \hilbert
  \rightarrow C$ which maps to an $x\in\hilbert$ the (unique)
  $P_C(x)\in C$ such that $\norm{x-P_C(x)} = d(x,C)$.

\begin{definition}[Subdifferential and subgradient]\label{def:subgradient}
Given a convex function $\Theta: \hilbert \rightarrow\Real$, the
subdifferential of $\Theta$ is defined as the set-valued mapping:
\begin{equation*}
\partial \Theta : \hilbert \rightarrow 2^{\hilbert} : x\mapsto
\partial\Theta(x):= \{v\in\hilbert:\ \forall y\in\hilbert,
\innprod{v}{y-x} + \Theta(x)\leq \Theta(y)\}.
\end{equation*}
In the case where $\Theta$ is continuous at $x$, then
$\partial\Theta(x)\neq \emptyset$ \cite{EkelandTemam}. Any element in
$\partial\Theta(x)$ will be called a \textit{subgradient of $\Theta$
  at $x$,} and will be denoted by $\Theta'(x)$. If $\Theta$ is
G\^{a}teaux differentiable at $x$, then $\partial\Theta(x)$ becomes a
singleton, and the unique element of $\partial\Theta(x)$ is nothing
but the classical G\^{a}teaux differential of $\Theta$ at $x$. Notice,
also, the well-known fact: $0\in \partial\Theta(x) \Leftrightarrow
x\in\argmin_{v\in\hilbert} \Theta(v)$.
\end{definition}

\begin{example}\label{ex:partial.distance}
The subdifferential of the metric distance function to a closed convex
set $C\subset\hilbert$ is given as follows:
\begin{equation*}
\partial d(x,C) = 
\begin{cases}
 N_C(x) \cap B[0,1], & \text{if}\ x\in C,\\ 
\frac{x-P_C(x)}{d(x,C)}, & \text{if}\ x\in \hilbert\setminus C,
\end{cases}
\end{equation*}
where $N_C(x):=\{v\in\hilbert:\ \forall y\in C, \innprod{v}{y-x}\leq
 0\}$. Notice that $\forall x\in\hilbert$, $\forall d'(x,C) \in \partial
 d(x,C)$, $\norm{d'(x,C)}\leq 1$.
\end{example}

\begin{definition}[\cite{BauschkeBorwein,
      BauschkeCombettesWeak2Strong, yo.hybrid.quasi.ne,
      bauschke.combettes.book}]
 \label{def:quasi-nonexpansive} Given a mapping
$T:\hilbert\rightarrow\hilbert$, the set of all fixed points of $T$,
 i.e., $\Fix(T) := \{v\in\hilbert:\ T(v) = v\}$, is called the
 \textit{fixed point set} of $T$. Assume a
 $T:\hilbert\rightarrow\hilbert$ such that $\Fix(T)\neq
 \emptyset$. The mapping $T$ will be called
 \textit{quasi-nonexpansive} if $\forall x\in\hilbert$, $\forall
 v\in\Fix(T)$, $\norm{T(x)- v} \leq \norm{x-v}$. It can be verified
 that the fixed point set of a quasi-nonexpansive mapping is closed
 and convex, e.g., \cite[Prop.\ 2.3 and
   2.6]{BauschkeCombettesWeak2Strong}. If
\begin{equation*}
\exists\eta>0:\ \forall x\in\hilbert, \forall
v\in\Fix(T), \quad \eta\norm{x-T(x)}^2 \leq \norm{x-v}^2 -
\norm{T(x)-v}^2,
\end{equation*}
then $T$ will be called \textit{$\eta$-attracting} or \textit{strongly
 attracting quasi-nonexpansive}.

Now, if $\forall x$, $y\in\hilbert$, $\norm{T(x)-T(y)} \leq \norm{x-y}$,
then $T$ will be called \textit{nonexpansive}. In the case where $T$ is
both nonexpansive and strongly attracting quasi-nonexpansive, then it
will be called \textit{strongly attracting nonexpansive}.

In particular, an $1$-attracting (quasi)-nonexpansive mapping will be
called \textit{firmly (quasi)-nonexpansive}.
\end{definition}

\begin{fact}[Equivalent description of strongly attracting
    quasi-nonexpansive mappings \cite{yo.hybrid.quasi.ne,
      vasin.ageev}]\label{fact:averaged.maps} The following statements
  are equivalent for a mapping $T:\hilbert\rightarrow \hilbert$.
\begin{enumerate}[1.]
\item $T$ is $\eta$-attracting quasi-nonexpansive.
\item $T$ is $\frac{1}{1+\eta}$-averaged quasi-nonexpansive. A mapping
  $T$ is called \textit{$\alpha$-averaged quasi-nonexpansive,} with
  $\alpha\in(0,1)$, if there exists a quasi-nonexpansive mapping
  $R:\hilbert \rightarrow\hilbert$ such that $T= (1-\alpha)I + \alpha R$.
\end{enumerate}
In particular, $T$ is firmly quasi-nonexpansive iff $T$ is
$\frac{1}{2}$-averaged quasi-nonexpansive. Notice that $\forall
\alpha\in (0,1)$, $\Fix(T)=\Fix(R)$, which suggests that given a
quasi-nonexpansive mapping $R$, we can always construct a strongly
attracting quasi-nonexpansive $T$ that shares the same fixed point set
with $R$.
\end{fact}

\begin{example}[Subgradient projection
    mapping]\label{def:subgrad.projection} Given a convex continuous
  function $\Theta$, such that $\lev\Theta :=
  \{v\in\hilbert:\ \Theta(v)\leq 0\} \neq \emptyset$, define the
  \textit{subgradient projection mapping $T_{\Theta}:
    \hilbert\rightarrow\hilbert$ with respect to $\Theta$} as follows:
\begin{equation*}
T_{\Theta}(x) := \begin{cases}
x - \frac{\Theta(x)}{\norm{\Theta'(x)}^2}
\Theta'(x), & \text{if}\ x\in\hilbert\setminus \lev\Theta,\\
x, & \text{if}\ x\in\lev\Theta,
\end{cases}
\end{equation*}
where $\Theta'(x)$ is any subgradient in $\partial\Theta(x)$. If $I$
stands for the identity mapping in $\hilbert$, the mapping
\begin{equation*}
T_{\Theta}^{(\lambda)}:= I + \lambda(T_{\Theta}- I),\quad \lambda\in(0,2),
\end{equation*}
will be called the \textit{relaxed subgradient projection mapping with
  respect to $\Theta$}. It can be verified that $\forall
\lambda\in(0,2)$, $\Fix(T_{\Theta}^{(\lambda)})= \Fix(T_{\Theta})
=\lev\Theta$ \cite{BauschkeCombettesWeak2Strong}. Moreover, $\forall
\lambda\in(0,2)$, the mapping $T_{\Theta}^{(\lambda)}$ is
$\frac{2-\lambda}{\lambda}$-attracting quasi-nonexpansive
\cite{BauschkeCombettesWeak2Strong}.
\end{example}

\begin{example}[Relaxed metric projection
    mapping]\label{ex:relaxed.projection} Let a nonempty closed convex
  set $C\subset\hilbert$ and its associated metric projection mapping
  $P_C$. Then, the relaxed (metric) projection mapping,
  $T_C^{(\alpha)}:= I + \alpha (P_C - I)$, $\alpha\in(0,2)$, is
  $\frac{2-\alpha}{\alpha}$-attracting nonexpansive with fixed point
  set $\Fix(T_C^{(\alpha)}) = C$ \cite{BauschkeBorwein}.
\end{example}

\begin{example}[\cite{YamadaOguraAPSMNFAO, BauschkeBorwein}]
\label{ex:compose.strongly.attracting}
Let $T_1, T_2$ be $\eta_1$- and
$\eta_2$-attracting (quasi)-nonexpansive mappings, respectively. Assume
also that $\Fix(T_1) \cap \Fix(T_2) \neq \emptyset$. Then, the mapping
$T_1T_2$ is $\frac{\eta_1\eta_2}{\eta_1 + \eta_2}$-attracting
(quasi)-nonexpansive, and $\Fix(T_1 T_2) = \Fix(T_1) \cap \Fix(T_2)$.
\end{example}

\begin{definition}[Demiclosed mapping at $0$]\label{def:demiclosed}
A mapping $T: \hilbert\rightarrow\hilbert$ will be called
\textit{demiclosed at $0$} if the following property holds; for a
sequence $(x_n)_{n\in\Natural}\subset\hilbert$, and an $x_*\in\hilbert$,
\begin{equation*}
\text{if}\ 
\begin{cases}
x_n \xrightharpoonup{n\rightarrow\infty} x_*,\\
T(x_n)\xrightarrow{n\rightarrow\infty} 0,
\end{cases}\ \text{then}\ T(x_*)=0,
\end{equation*}
where the symbols $\rightharpoonup$ and $\rightarrow$ denote weak and
strong convergence in $\hilbert$, respectively.
\end{definition}

\begin{example}[{\cite[Lem.\ 2]{Opial}}]\label{ex:nonexpansive.is.demiclosed}
If $T:\hilbert\rightarrow\hilbert$ is a nonexpansive mapping, then
$I-T$ is demiclosed at $0$.
\end{example}

\begin{example}[{\cite[Prop.\ 6.10]{BauschkeCombettesWeak2Strong}},
    \cite{vasin.ageev}]\label{ex:subgrad.projection.is.demiclosed.at.zero} 
Let a continuous convex function $\Theta: \hilbert \rightarrow \Real$
such that $\lev\Theta \neq\emptyset$. Then, $\forall\lambda\in(0,2)$,
the mapping $I-T_{\Theta}^{(\lambda)}$ is demiclosed at $0$, where
$T_{\Theta}^{(\lambda)}$ stands for the relaxed subgradient projection
mapping with respect to $\Theta$.
\end{example}

\begin{fact}[\cite{YamadaOguraAPSMNFAO}]\label{fact:special.Fejer}
Assume a sequence $(x_n)_{n\in\Natural}\subset\hilbert$, and a
closed convex set $C\subset\hilbert$. Assume that
\begin{equation*}
\exists\kappa>0:\ \forall v\in C,\ \forall n\in\Natural, \quad
\kappa\norm{x_{n+1} - x_n}^2 \leq \norm{x_n - v}^2 -
\norm{x_{n+1} - v}^2.
\end{equation*}
If there exists, also, a hyperplane $\Pi$ such that
$\relinterior_{\Pi}C\neq \emptyset$, then $\exists
x_*\in\hilbert$ such that $x_*=\lim_{n\rightarrow\infty} x_n$.
\end{fact}

\begin{definition}[Inner and outer limits \cite{Rockafellar.Wets,
      aubin.frankowska}]\label{def:inner.outer}
Given a sequence of subsets $(S_n)_{n\in\Natural}\subset \hilbert$,
define the inner and outer limits:
\begin{align}
\liminf_{n\rightarrow\infty} S_n & := \left\{x\in\hilbert: \exists N\in
\Natural_{\infty}, \exists x_n\in S_n, \forall n\in N,\ \text{such
  that}\ \lim_{n\in N}x_n=x\right\} \nonumber\\
& = \left\{x\in\hilbert: \limsup_{n\rightarrow\infty}
d(x,S_n)=0\right\} = \bigcap_{N\in \Natural_{\infty}^{\#}}
\overline{\bigcup_{n\in N} S_n}\label{liminf.def.1}\\
& = \bigcap_{\epsilon>0} \left[\bigcup_{n=1}^{\infty} \bigcap_{k=n}^{\infty}
\left(S_k+ B[0,\epsilon]\right)\right],\label{liminf.def.2}\\
\limsup_{n\rightarrow\infty} S_n & := \left\{x\in\hilbert: \exists N\in
\Natural_{\infty}^{\#}, \exists x_n\in S_n, \forall n\in N,\ \text{such
  that}\ \lim_{n\in N} x_n=x\right\} \nonumber\\
& = \left\{x\in\hilbert: \liminf_{n\rightarrow\infty}
d(x,S_n)=0\right\} = \bigcap_{N\in \Natural_{\infty}}
\overline{\bigcup_{n\in N} S_n}\nonumber\\
& = \bigcap_{\epsilon>0} \left[\bigcap_{n=1}^{\infty} \bigcup_{k=n}^{\infty}
\left(S_k+ B[0,\epsilon]\right)\right],\nonumber
\end{align}
where $S_k + B[0,\epsilon]:= \{s+b: s\in S_k, b\in B[0,\epsilon]\}$,
and the overline symbol stands for the closure of a set. In a similar
fashion, given a sequence of subsets $(S_n)_{n\in \Natural}$, and a
subsequence $N=(n_k)_{k\in\Natural}\in \Natural_{\infty}^{\#}$, the
notation $\liminf_{n\in N}S_n$ is defined as
$\liminf_{k\rightarrow\infty} S_{n_k}$. Likewise, $\limsup_{n\in
  N}S_n:= \limsup_{k\rightarrow\infty} S_{n_k}$.

\end{definition}

\section{The Analysis of the Algorithm}\label{sec:algo}

\subsection{A useful theorem.}\label{sec:useful.thm}

Prior to the analysis of Algorithm~\ref{the.algo}, we state and prove
Theorem~\ref{thm:liminf}, which will be repeatedly used in the
sequel. The proof of Theorem~\ref{thm:liminf} will be based on the
following assumption.

\begin{assumption}\label{ass:liminf.Fn}
Assume a sequence of mappings $(T_n:
\hilbert\rightarrow \hilbert)_{n\in\Natural}$ with nonempty fixed
point sets $(\Fix(T_n))_{n\in\Natural}$. For any subsequence $N\in
\Natural_{\infty}^{\#}$, for any sequence $(x_n)_{n\in
  N}\subset\hilbert$, and for any $\gamma>0$ such that $\forall n\in
N$, $d(x_n,\Fix(T_n))\geq \gamma$, there exists a $\delta>0$ such that
$\liminf_{n\in N} \norm{(I-T_n)(x_n)} \geq\delta$.
\end{assumption}

\begin{theorem}\label{thm:liminf}
Assume a sequence of mappings $(T_n: \hilbert\rightarrow
\hilbert)_{n\in\Natural}$, with nonempty fixed point sets
$(\Fix(T_n))_{n\in\Natural}$, such that Assumption~\ref{ass:liminf.Fn}
is satisfied.

\begin{enumerate}[1.]

\item\label{thm:liminf2} Assume a subsequence $N\in
  \Natural_{\infty}^{\#}$, a sequence $(x_n)_{n\in N}
  \subset\hilbert$ and an $x_*\in\hilbert$.
\begin{equation*}
\text{If}\ \begin{cases}
x_n \xrightarrow{n\in N} x_*,\\ 
(I-T_n)(x_n) \xrightarrow{n\in N} 0,
\end{cases}\ \text{then}\ x_*\in
\liminf_{n\in N} \Fix(T_n).
\end{equation*}

\item\label{thm:limsup} Let $\mathfrak{S}((x_n)_{n\in\Natural})$ be
  the set of all strong cluster points of a sequence
  $(x_n)_{n\in\Natural}$.
\begin{equation*}
\text{If}\ \begin{cases}
\mathfrak{S}((x_n)_{n\in\Natural})\neq \emptyset,\\ 
(I-T_{n})(x_n) \xrightarrow{n\rightarrow\infty} 0,
\end{cases}\ \text{then}\ \mathfrak{S}((x_n)_{n\in\Natural}) \subset
\limsup_{n\rightarrow\infty} \Fix(T_n).
\end{equation*}

\end{enumerate}
\end{theorem}

\begin{proof}
\begin{enumerate}[1.]

\item We will prove Theorem~\ref{thm:liminf}.\ref{thm:liminf2} by
  contradiction, i.e., assume that
  $x_*\notin \liminf_{n\in N} \Fix(T_n)$.

By \eqref{liminf.def.1}, $\limsup_{n\in N}d(x_*,\Fix(T_n))>0$, i.e.,
there exists $\tau >0$, and $\exists N'\in\Natural_{\infty}^{\#}$,
such that $\forall n\in N'\cap N$, $d(x_*, \Fix(T_n))>\tau$.

Moreover, since $\lim_{n\rightarrow\infty} x_n = x_*$, there exists an $N_0\in
\Natural_{\infty}$ such that $\forall n\in N_0\cap N$, $\norm{x_* - x_n} <
\frac{\tau}{2}$. Having these in mind, the triangle inequality
$\norm{x_*- v}\leq \norm{x_*- x_n} + \norm{x_n - v}$, $\forall
v\in\Fix(T_n)$, leads us to the following:
\begin{equation*}
\forall n\in N_0 \cap N'\cap N, \quad d(x_n,\Fix(T_n)) \geq
d(x_*,\Fix(T_n)) - \norm{x_*- x_n} > \tau - \frac{\tau}{2} =:
\gamma>0.
\end{equation*}
Hence, there exists a subsequence $N'':= N_0\cap N'\cap N\in
\Natural_{\infty}^{\#}$ such that $\forall n\in N''$,
$d(x_n,\Fix(T_n))\geq\gamma$.

Now, by Assumption~\ref{ass:liminf.Fn}, there exists a $\delta>0$ such
that 
\begin{equation*}
0< \delta\leq \liminf_{n\in N''} \norm{(I-T_n)(x_n)} = \lim_{n\in
  N''} \norm{(I-T_n)(x_n)} = 0,
\end{equation*}
where the last two equalities come
from the fact that $N''\subset N$. This contradiction establishes
Theorem~\ref{thm:liminf}.\ref{thm:liminf2}.

\item Choose arbitrarily an $x_* \in
  \mathfrak{S}((x_n)_{n\in\Natural})$. By definition, there exists a
  subsequence $N\in\Natural_{\infty}^{\#}$ such that $\lim_{n\in N}x_n
  = x_*$. Hence, by Theorem~\ref{thm:liminf}.\ref{thm:liminf2},
  $x_*\in \liminf_{n\in N} \Fix(T_n)$. By
  Definition~\ref{def:inner.outer}, $\exists N_0\in \Natural_{\infty}$
  and $\exists x_n'\in \Fix(T_n)$, $\forall n\in N\cap N_0$ such that
  $\lim_{n\in N\cap N_0} x_n' = x_*$.

  Clearly, $N':= N\cap N_0\in \Natural_{\infty}^{\#}$. In other words,
  $\exists N'\in \Natural_{\infty}^{\#}$, $\exists x_n'\in \Fix(T_n)$,
  $\forall n\in N'$ such that $\lim_{n\in N'} x_n' = x_*$, i.e.,
  $x_*\in \limsup_{n\rightarrow\infty}\Fix(T_n)$ by
  Definition~\ref{def:inner.outer}. Since $x_*$ was chosen
  arbitrarily, Theorem~\ref{thm:liminf}.\ref{thm:limsup} is
  established.\qedhere

\end{enumerate}
\end{proof}

Next is an example of a sequence of mappings which satisfies
Assumption~\ref{ass:liminf.Fn}, and which will be used later on in the
sequel. Another example of a family of mappings which satisfies
Assumption~\ref{ass:liminf.Fn}, and which relates to the minimization
of an $\ell_1$-norm loss function, will be seen in
Lemma~\ref{lem:l1.function}.\ref{lem:l1.liminf.ass}.

\begin{example}\label{ex:distance.function}
Assume a sequence of nonempty closed convex sets
$(S_n)_{n\in\Natural}$, the associated sequence of relaxed metric
projection mappings 
\begin{equation*}
T_{S_n}^{(\alpha_n)}:= I + \alpha_n(P_{S_n}-I), \quad \alpha_n\in(0,2),
\forall n\in\Natural,
\end{equation*}
and the existence of a sufficiently small $\epsilon>0$ such that
$\alpha_n\in [\epsilon,2)$, $\forall n\in\Natural$. Then, the sequence
of mappings $(T_{S_n}^{(\alpha_n)})_{n\in\Natural}$ satisfies
Assumption~\ref{ass:liminf.Fn}.
\end{example}

\begin{proof}
First of all, by Example~\ref{ex:relaxed.projection}, $\forall
n\in\Natural$, $\Fix(T_{S_n}^{(\alpha_n)})=S_n$. Choose, now,
arbitrarily an $N\in \Natural_{\infty}^{\#}$, a sequence $(x_n)_{n\in
  N}\subset \hilbert$, and a $\gamma>0$, such that $\forall n\in N$,
$d(x_n, \Fix(T_{S_n}^{(\alpha_n)}))= d(x_n,S_n) \geq \gamma$. Then, it
is easy to verify by the definition of $T_{S_n}^{(\alpha_n)}$ that
\begin{equation*}
\forall n\in N, \quad \norm{(I-T_{S_n}^{(\alpha_n)})(x_n)} = \alpha_n
d(x_n,S_n)\geq \epsilon\gamma>0. 
\end{equation*}
Therefore, there exists a $\delta>0$ such that $\liminf_{n\in N}
\norm{(I- T_{S_n}^{(\alpha_n)})(x_n)} \geq \delta$, and
Assumption~\ref{ass:liminf.Fn} is established.
\end{proof}

\subsection{The Main Analysis}\label{sec:main.analysis}

Given a sequence of convex, continuous, and not necessarily
differentiable functions $(\Theta_n: \hilbert\rightarrow
[0,\infty))_{n\in\Natural}$, and a sequence of $\eta_n$-attracting
  quasi-nonexpansive mappings $(T_n: \hilbert \rightarrow
  \hilbert)_{n\in\Natural}$, with $\eta_n>0$, $\forall n\in\Natural$,
  and with nonempty fixed point sets $(\Fix(T_n))_{n\in\Natural}$, the
  convergence analysis of Algorithm~\ref{the.algo}, given in 
Theorem~\ref{thm:the.analysis}, will be based on the following series
of assumptions.

\begin{assumption}\label{main.assumptions}\mbox{}
\begin{enumerate}[1.]

\item\label{ass:nonempty.Omega_n} There exists an $N\in \Natural_{\infty}$
  such that $\forall n\in N$, $\Omega_n:= \Fix(T_n) \cap
  \lev\Theta_n \neq\emptyset$.

\item\label{ass:nonempty.intersection} There exists an
  $N\in\Natural_{\infty}$ such that $\Omega := \bigcap_{n\in N}
  \Omega_n \neq\emptyset$.

\item\label{ass:lambda.epsilon} Choose an 
$\epsilon\in(0,1]$, and let $\forall n\in\Natural$, $\lambda_n\in
  [\epsilon,2-\epsilon]$.

\item\label{ass:bounded.subgradients} The sequence
  $(\Theta_n'(u_n))_{n\in\Natural}$ is bounded.

\item\label{ass:eta} Define $\check{\eta} := \inf\{\eta_n:\
     n\in\Natural\}$, $\hat{\eta} := \sup\{\eta_n:\
     n\in\Natural\}$. Then, assume that $\check{\eta} >0$ and
     $\hat{\eta}<\infty$.

\item\label{ass:extra.on.Theta_n} The sequence of relaxed
  subgradient projection mappings
  $(T_{\Theta_n}^{(\lambda_n)})_{n\in\Natural}$ satisfies
  Assumption~\ref{ass:liminf.Fn}.  

\item\label{ass:extra.on.T_n} The sequence of mappings
  $(T_n)_{n\in\Natural}$ satisfies Assumption~\ref{ass:liminf.Fn}.

\item\label{ass:demiclosed} Assume that $\forall n\in\Natural$, $T_n:=
  T$, where $T$ is a strongly attracting quasi-nonexpansive mapping
  with $\Fix(T)\neq \emptyset$, and $I-T$ is demiclosed at $0$.

\item\label{ass:strong.clusters} The set
  $\mathfrak{S}((u_n)_{n\in\Natural})$ of all strong cluster points of
  the sequence $(u_n)_{n\in\Natural}$ is nonempty.

\item\label{ass:ri.Omega.hyperplane} There exists a hyperplane $\Pi$ such
  that $\relinterior_{\Pi}(\Omega) \neq \emptyset$.

\end{enumerate}
\end{assumption}

\begin{theorem}[Properties of
    Algorithm~\ref{the.algo}]\label{thm:the.analysis}
\mbox{}

\begin{enumerate}[1.]

\item\label{thm:monotone.Omega_n} Let
  Assumption~\ref{main.assumptions}.\ref{ass:nonempty.Omega_n} hold
  true. Then, $\forall n\in N$, $d(u_{n+1},\Omega_n) \leq
  d(u_n,\Omega_n)$.

\item\label{thm:monotone.Omega} Let
  Assumption~\ref{main.assumptions}.\ref{ass:nonempty.intersection}
  hold true. Then, $\forall n\in N$, $d(u_{n+1},\Omega) \leq
  d(u_n,\Omega)$. 

\item\label{thm:convergent.sequence} Let
  Assumption~\ref{main.assumptions}.\ref{ass:nonempty.intersection}
  hold true. Then, $\forall v\in\Omega$, the sequence
  $(\norm{u_n-v})_{n\in\Natural}$ converges. 

\item\label{thm:nonempty.weak.cluster.points} Let
  Assumption~\ref{main.assumptions}.\ref{ass:nonempty.intersection}
  hold true. Then, the set of all weakly sequential cluster points of
  the sequence $(u_n)_{n\in\Natural}$ is nonempty, i.e.,
  $\mathfrak{W}((u_n)_{n\in\Natural})\neq \emptyset$.

\item\label{thm:fraction.goes.to.zero} Let
  Assumptions~\ref{main.assumptions}.\ref{ass:nonempty.intersection} and
  \ref{main.assumptions}.\ref{ass:lambda.epsilon} hold true. Then,
\begin{equation*}
\lim_{n\rightarrow\infty} \norm{(I-T_{\Theta_n}^{(\lambda_n)})(u_n)} =
\lim_{n\rightarrow\infty} \frac{\Theta_n(u_n)}{\norm{\Theta'_n(u_n)}}
=0,
\end{equation*}
where, in order to avoid ambiguities, we let $\frac{0}{0} :=0$.

\item\label{asymptotic.optimality} Let
  Assumptions~\ref{main.assumptions}.\ref{ass:nonempty.intersection},
  \ref{main.assumptions}.\ref{ass:lambda.epsilon}, and
  \ref{main.assumptions}.\ref{ass:bounded.subgradients} hold
  true. Then, $\lim_{n\rightarrow\infty} \Theta_n(u_n)=0$.

\item\label{thm:liminf.Psi} Let
  Assumptions~\ref{main.assumptions}.\ref{ass:nonempty.intersection},
  \ref{main.assumptions}.\ref{ass:lambda.epsilon},
  \ref{main.assumptions}.\ref{ass:extra.on.Theta_n}, and
  \ref{main.assumptions}.\ref{ass:strong.clusters} hold true. Then,
  $\mathfrak{S}((u_n)_{n\in\Natural})\subset
  \limsup_{n\rightarrow\infty} \lev\Theta_n$. If, in addition, the set
  $\mathfrak{S}((u_n)_{n\in\Natural})$ is a singleton, i.e., there
  exists a $u_*$ such that $\{u_*\} =
  \mathfrak{S}((u_n)_{n\in\Natural})$, then, $u_*\in
  \liminf_{n\rightarrow\infty}\lev\Theta_n$.  

\item\label{thm:T_n.demiclosed} Let
  Assumptions~\ref{main.assumptions}.\ref{ass:nonempty.intersection}
  and \ref{main.assumptions}.\ref{ass:eta} hold true. Then,
  $\lim_{n\rightarrow\infty}(I- T_n)(T_{\Theta_n}^{(\lambda_n)}(u_n))
  = 0$.

\item\label{thm:liminf.T_n} Let
  Assumptions~\ref{main.assumptions}.\ref{ass:nonempty.intersection},
  \ref{main.assumptions}.\ref{ass:lambda.epsilon},
  \ref{main.assumptions}.\ref{ass:eta},
  \ref{main.assumptions}.\ref{ass:extra.on.T_n}, and
  \ref{main.assumptions}.\ref{ass:strong.clusters} hold true. Then,
  $\mathfrak{S}((u_n)_{n\in\Natural})\subset
  \limsup_{n\rightarrow\infty}\Fix(T_n)$. If, in addition, the set
  $\mathfrak{S}((u_n)_{n\in\Natural})$ is a singleton, i.e., there
  exists a $u_*$ such that $\{u_*\} =
  \mathfrak{S}((u_n)_{n\in\Natural})$, then, $u_*\in
  \liminf_{n\rightarrow\infty}\Fix(T_n)$.  

\item\label{thm:weak.cluster.points.in.fixT} Let
  Assumptions~\ref{main.assumptions}.\ref{ass:nonempty.intersection},
  \ref{main.assumptions}.\ref{ass:lambda.epsilon}, and
  \ref{main.assumptions}.\ref{ass:demiclosed} hold true. Then,
  $\mathfrak{W}((u_n)_{n\in\Natural}) \subset\Fix(T)$. 

\item\label{thm:strong.cluster.points.in.fixT} Let
  Assumptions~\ref{main.assumptions}.\ref{ass:nonempty.intersection},
  \ref{main.assumptions}.\ref{ass:lambda.epsilon},
  \ref{main.assumptions}.\ref{ass:demiclosed}, and
  \ref{main.assumptions}.\ref{ass:strong.clusters} hold true. Then,
  $\mathfrak{S}((u_n)_{n\in\Natural}) \subset\Fix(T)$.

\item\label{thm:strong.convergence} Let
  Assumptions~\ref{main.assumptions}.\ref{ass:nonempty.intersection},
  \ref{main.assumptions}.\ref{ass:lambda.epsilon},
  \ref{main.assumptions}.\ref{ass:eta}, and
  \ref{main.assumptions}.\ref{ass:ri.Omega.hyperplane} hold
  true. Then, $\exists
  u_*\in\hilbert:\ \lim_{n\rightarrow\infty}u_n=u_*$, i.e., $
  \mathfrak{S}((u_n)_{n\in\Natural})=\{u_*\}$.

\end{enumerate}
\end{theorem}

\begin{proof}
\begin{enumerate}[1.]

\item By assumption~\ref{main.assumptions}.\ref{ass:nonempty.Omega_n},
  $\forall n\in N$, $\lev\Theta_n \neq 
  \emptyset$. Recall also the fundamental fact that $0\in
  \partial\Theta_n(u_n) \Leftrightarrow u_n\in\argmin_{v\in\hilbert}
  \Theta_n(v)$.

  Fix any $n\in N$. Consider the case where $u_n\notin \lev\Theta_n
  \Leftrightarrow \Theta_n(u_n)>0 \Rightarrow \Theta_n'(u_n)\neq
  0$. Then, by \eqref{basic.recursion}, $u_{n+1}=T_n\left(u_n - \lambda_n
  \frac{\Theta_n(u_n)}{\norm{\Theta_n'(u_n)}^2} \Theta_n'(u_n)
  \right)$. Now, assume that $u_n\in \lev\Theta_n \Leftrightarrow
  \Theta_n(u_n)=0$. If $\Theta_n'(u_n)=0$, then by
  \eqref{basic.recursion}, $u_{n+1}= T_n(u_n)$. On the other hand, if
  $\Theta_n'(u_n)\neq 0$, then, again, $u_{n+1}=T_n(u_n)$, since
  $\Theta_n(u_n)=0$. To summarize, \eqref{basic.recursion} takes the
  following form:
  \begin{equation*}
    \forall n\in N,\quad u_{n+1} := \begin{cases}
      T_n\left(u_n - \lambda_n
      \frac{\Theta_n(u_n)}{\norm{\Theta_n'(u_n)}^2}
      \Theta_n'(u_n) \right), & \text{if}\ u_n\notin\lev\Theta_n,\\
      T_n(u_n), & \text{if}\ u_n\in\lev\Theta_n.
    \end{cases}
  \end{equation*}
  If we combine this result with
  Example~\ref{def:subgrad.projection}, then it can be easily
  verified that the previous recursion can be equivalently viewed as
  follows: $\forall n\in N$, $u_{n+1} = T_n
  T_{\Theta_n}^{(\lambda_n)}(u_n)$, where $T_{\Theta_n}^{(\lambda_n)}$
  stands for the relaxed subgradient projection mapping
  w.r.t.\ $\Theta_n$.

  Now, since $T_{\Theta_n}^{(\lambda_n)}$ is a
  $\frac{2-\lambda_n}{\lambda_n}$-attracting quasi-nonexpansive
  mapping, with $\Fix(T_{\Theta_n}^{(\lambda_n)})= \lev\Theta_n$, it
  can be easily verified by
  Example~\ref{ex:compose.strongly.attracting} that the mapping $T_n
  T_{\Theta_n}^{(\lambda_n)}$ is
  $\frac{(2-\lambda_n)\eta_n}{2-\lambda_n(1-\eta_n)}$-attracting
  quasi-nonexpansive, with $\Fix(T_n T_{\Theta_n}^{(\lambda_n)})=
  \Fix(T_n)\cap \Fix(T_{\Theta_n}^{(\lambda_n)}) = \Fix(T_n)\cap
  \lev\Theta_n = \Omega_n$, $\forall n\in N$. Hence, by
  Definition~\ref{def:quasi-nonexpansive}, we have that $\forall n\in
  N$, $\forall v\in\Omega_n$,
\begin{align}
0 & \leq \frac{(2-\lambda_n)\eta_n}{2-\lambda_n(1-\eta_n)} \norm{u_n -
  u_{n+1}}^2 = \frac{(2-\lambda_n)\eta_n}{2-\lambda_n(1-\eta_n)}
\norm{u_n - T_n T_{\Theta_n}^{(\lambda_n)}(u_n)}^2 \nonumber\\
& \leq \norm{u_n - v}^2 - \norm{T_nT_{\Theta_n}^{(\lambda_n)}(u_n) -
  v}^2 = \norm{u_n - v}^2  -
\norm{u_{n+1} - v}^2\label{strongly.attracting.Omega_n}\\
& \Rightarrow \norm{u_{n+1} - v} \leq \norm{u_n -
  v}.\label{monotonicity}
\end{align}
If we apply $\inf_{v\in\Omega_n}$, on both sides of
\eqref{monotonicity}, then we obtain
Theorem~\ref{thm:the.analysis}.\ref{thm:monotone.Omega_n}.

\item Due to
  Assumption~\ref{main.assumptions}.\ref{ass:nonempty.intersection},
  to the fact that $\Omega$ is closed and convex, to
  $P_{\Omega}(u_n)\in \Omega\subset\Omega_n$, $\forall n\in N$, and to
  \eqref{monotonicity}, we have:
\begin{align}
\forall n\in N, \quad d(u_n,\Omega) & = \norm{u_n
  - P_{\Omega}(u_n)} \geq \norm{u_{n+1} - P_{\Omega}(u_n)} \nonumber\\
& \geq \norm{u_{n+1} - P_{\Omega}(u_{n+1})} =
d(u_{n+1},\Omega), \label{monotonicity.Omega}
\end{align}
which is nothing but
Theorem~\ref{thm:the.analysis}.\ref{thm:monotone.Omega}.

\item Fix arbitrarily $v\in\Omega$. By \eqref{monotonicity}, the
  sequence $(\norm{u_n-v})_{n\in N}$ is non-increasing and
  bounded; hence convergent. This establishes
  Theorem~\ref{thm:the.analysis}.\ref{thm:convergent.sequence}.

\item Since $(u_n)_{n\in\Natural}$ is bounded by
  Theorem~\ref{thm:the.analysis}.\ref{thm:convergent.sequence},
  $\mathfrak{W}((u_n)_{n\in\Natural})\neq \emptyset$
  \cite[Thm.\ 9.12]{DeutschBook}. This establishes
  Theorem~\ref{thm:the.analysis}.\ref{thm:nonempty.weak.cluster.points}.

\item There is no loss of generality if we assume that $\forall
  n\in N$, $\Theta_n'(u_n)\neq 0$. To see this, notice that for
  all $n\in N$ such that $\Theta'_n(u_n) = 0$, we obtain
  $\Theta_n(u_n) = 0 \Rightarrow
  \frac{\Theta_n(u_n)}{\norm{\Theta'_n(u_n)}} = \frac{0}{0} :=0$. Hence,
  in such a case, the claim of
  Theorem~\ref{thm:the.analysis}.\ref{thm:fraction.goes.to.zero} holds
  true.

Assume, now, any $v\in\Omega$. Recall also that the mapping $T_n$ is
quasi-nonexpansive, with $\Omega\subset \Fix(T_n)$, $\forall n\in
N$, and easily verify $\forall n\in N$, $\forall v\in\Omega$,
\begin{align}
\norm{u_{n+1} - v}^2 & = \norm{ T_{n}\left(u_{n} -
\lambda_{n}
\frac{\Theta_{n}(u_{n})}{\norm{\Theta_{n}'(u_{n})}^2}
\Theta_{n}'(u_{n}) \right) - v}^2 \leq \norm{u_{n} - \lambda_{n}
\frac{\Theta_{n}(u_{n})}{\norm{\Theta_{n}'(u_{n})}^2}
\Theta_{n}'(u_{n}) - v}^2\nonumber\\
& = \norm{(u_{n}-v) - \lambda_{n}
\frac{\Theta_{n}(u_{n})}{\norm{\Theta_{n}'(u_{n})}^2}
\Theta_{n}'(u_{n})}^2\nonumber\\
& = \norm{u_{n} - v}^2 + \lambda_{n}^2
\frac{\Theta_{n}^2(u_{n})}{\norm{\Theta_{n}'(u_{n})}^2}
- 2 \lambda_{n}
\frac{\Theta_{n}(u_{n})}{\norm{\Theta_{n}'(u_{n})}^2}
\innprod{u_{n} - v}{\Theta_{n}'(u_{n})}.
\label{before.using.subgrad}
\end{align}
By the definition of the subgradient, we have that
$\innprod{v-u_{n}}{\Theta_{n}'(u_{n})} + \Theta_{n}(u_{n})
\leq \Theta_{n}(v) = 0$. If we merge this into
\eqref{before.using.subgrad}, we obtain the following:
\begin{equation*}
\norm{u_{n+1} - v}^2 \leq \norm{u_{n} - v}^2 +
 \lambda_{n}^2
 \frac{\Theta_{n}^2(u_{n})}{\norm{\Theta_{n}'(u_{n})}^2}
 - 2 \lambda_{n}
 \frac{\Theta_{n}^2(u_{n})}{\norm{\Theta_{n}'(u_{n})}^2} 
 = \norm{u_{n} - v}^2 -
 \lambda_{n} (2-\lambda_{n})
 \frac{\Theta_{n}^2(u_{n})}{\norm{\Theta_{n}'(u_{n})}^2}.
\end{equation*}
This implies in turn that
\begin{equation}
\forall n\in N, \forall v\in\Omega, \quad 0\leq
\frac{\Theta_{n}^2(u_{n})}{\norm{\Theta_{n}'(u_{n})}^2} \leq
\frac{\lambda_{n} (2-\lambda_{n})}{\epsilon^2}
\frac{\Theta_{n}^2(u_{n})}{\norm{\Theta_{n}'(u_{n})}^2} \nonumber \leq
\frac{\norm{u_{n} - v}^2 - \norm{u_{n+1} -
    v}^2}{\epsilon^2}.\label{root.of.contradiction}
\end{equation}
However, by
Theorem~\ref{thm:the.analysis}.\ref{thm:convergent.sequence}, the 
sequence $(\norm{u_n - v}^2)_{n\in\Natural}$ is convergent, and hence
Cauchy. The definition of a Cauchy sequence implies
that $\lim_{n\rightarrow\infty} (\norm{u_{n} - v}^2 -
\norm{u_{n+1} - v}^2)=0$. This fact and the previous inequality
establish
$\lim_{n\rightarrow\infty}\frac{\Theta_n(u_n)}{\norm{\Theta_n'(u_n)}}
= 0$.

Now, notice that for all $n\in N$:
\begin{equation*}
\norm{u_n - T_{\Theta_n}^{(\lambda_n)}(u_n)} = \lambda_n
\frac{\Theta_n(u_n)}{\norm{\Theta_n'(u_n)}} \leq 2
\frac{\Theta_n(u_n)}{\norm{\Theta_n'(u_n)}}.
\end{equation*}
Take $\lim_{n\rightarrow\infty}$ on both sides of this inequality, and
recall the previous result to easily verify that
\begin{equation*}
\lim_{n\rightarrow\infty} \norm{u_n -
  T_{\Theta_n}^{(\lambda_n)}(u_n)} = 0.
\end{equation*}
In other words,
Theorem~\ref{thm:the.analysis}.\ref{thm:fraction.goes.to.zero} holds
true.

\item Since the sequence $(\Theta_n'(u_n))_{n\in\Natural}$ is assumed
  bounded, there exists a $D>0$ such that $\forall n\in\Natural$,
  $\norm{\Theta_n'(u_n)}\leq D$. Notice, now, that for all those
  $n\in\Natural$ such that $\Theta_n'(u_n)\neq 0$, we have
\begin{equation}
\Theta_n(u_n) = \norm{\Theta_n'(u_n)}
\frac{\Theta_n(u_n)}{\norm{\Theta_n'(u_n)}} \leq D
\frac{\Theta_n(u_n)}{\norm{\Theta_n'(u_n)}}.
\label{bounded.subgradient.prio.taking.the.limit}  
\end{equation}
Moreover, for all those $n\in N$ such that $\Theta_n'(u_n)= 0$,
it is clear by the well-known fact $0\in\partial\Theta_n(u_n)
\Leftrightarrow u_n\in\argmin_{v\in\hilbert} \Theta_n(v)$, that
$\Theta_n(u_n)=0$. If we take $\lim_{n\rightarrow\infty}$ on
both sides of \eqref{bounded.subgradient.prio.taking.the.limit}, and
if we also recall
Theorem~\ref{thm:the.analysis}.\ref{thm:fraction.goes.to.zero}, the
claim is established.

\item Notice that $\forall n\in N$, $\Fix(T_{\Theta_n}^{(\lambda_n)})
  = \lev\Theta_n$. Hence, $\mathfrak{S}((u_n)_{n\in\Natural}) \subset
  \limsup_{n\rightarrow\infty} \lev\Theta_n$ is a direct consequence
  of Theorems~\ref{thm:liminf} and
  \ref{thm:the.analysis}.\ref{thm:fraction.goes.to.zero}. The claim
  for the case of $\mathfrak{S}((u_n)_{n\in\Natural}) = \{u_*\}$ can
  be easily obtained if we let $N:=\Natural$ in
  Theorem~\ref{thm:liminf}.\ref{thm:liminf2}.

\item Here we will use Definition~\ref{def:quasi-nonexpansive} two
  times; one for the mapping $T_n$, and one for
  $T_{\Theta_n}^{(\lambda_n)}$. In other words, $\forall n\in N$,
  $\forall v\in\Omega$,
\begin{align*}
\check{\eta} \norm{(I-T_n) (T_{\Theta_n}^{(\lambda_n)}(u_n))}^2 & =
 \check{\eta} \norm{T_{\Theta_n}^{(\lambda_n)}(u_n) -
 T_nT_{\Theta_n}^{(\lambda_n)}(u_n)}^2 \leq \eta_n
 \norm{T_{\Theta_n}^{(\lambda_n)}(u_n) -
 T_nT_{\Theta_n}^{(\lambda_n)}(u_n)}^2 \\ & \leq
 \norm{T_{\Theta_n}^{(\lambda_n)}(u_n) - v}^2 -
 \norm{T_nT_{\Theta_n}^{(\lambda_n)}(u_n) - v}^2 \\ & =
 \norm{T_{\Theta_n}^{(\lambda_n)}(u_n) - v}^2 - \norm{u_{n+1}
 - v}^2 \\ & \leq \norm{u_n - v}^2 -
 \frac{2-\lambda_n}{\lambda_n} \norm{u_n -
 T_{\Theta_n}^{(\lambda_n)}(u_n)}^2 -
 \norm{u_{n+1} - v}^2\\ & \leq \norm{u_n - v}^2 -
 \norm{u_{n+1} - v}^2.
\end{align*}
Divide the above inequality by $\check{\eta}>0$,
recall Theorem~\ref{thm:the.analysis}.\ref{thm:convergent.sequence},
and take $\lim_{n\rightarrow\infty}$ on both sides of the resulting inequality
to obtain $\lim_{n\rightarrow\infty}
(I-T_n)(T_{\Theta_n}^{(\lambda_n)}(u_n)) = 0$. This establishes
Theorem~\ref{thm:the.analysis}.\ref{thm:T_n.demiclosed}.

\item First, since $\mathfrak{S}((u_n)_{n\in\Natural}) \neq
  \emptyset$, notice that
  $\mathfrak{S}((T_{\Theta_n}^{(\lambda_n)}(u_n))_{n\in\Natural}) =
  \mathfrak{S}((u_n)_{n\in\Natural})$. To establish, for example,
  $\mathfrak{S}((u_n)_{n\in\Natural}) \subset
  \mathfrak{S}((T_{\Theta_n}^{(\lambda_n)}(u_n))_{n\in\Natural})$,
  choose arbitrarily a $u_* \in \mathfrak{S}((u_n)_{n\in\Natural})$,
  which implies that there exists a subsequence
  $N'\in\Natural_{\infty}^{\#}$ such that $\lim_{n\in
    N'}u_n=u_*$. Then, it is easy to verify that
\begin{equation*}
\forall n\in N', \quad \norm{u_* -
  T_{\Theta_n}^{(\lambda_n)}(u_n)} \leq \norm{u_* - u_n} +
\norm{(I-T_{\Theta_n}^{(\lambda_n)})(u_n)}.
\end{equation*}
Take $\lim_{n\in N'}$ on both sides of the previous inequality, so
that the following result is obtained by
Theorem~\ref{thm:the.analysis}.\ref{thm:fraction.goes.to.zero}: $u_*
\in
\mathfrak{S}((T_{\Theta_n}^{(\lambda_n)}(u_n))_{n\in\Natural})$. Similar
arguments can be used in order to derive
$\mathfrak{S}((T_{\Theta_n}^{(\lambda_n)}(u_n))_{n\in\Natural})
\subset \mathfrak{S}((u_n)_{n\in\Natural})$.

Now, it becomes clear under the previous discussion, that if we define
$x_n:= T_{\Theta_n}^{(\lambda_n)}(u_n)$, $\forall n\in\Natural$, in
Theorem~\ref{thm:liminf}, then
Theorem~\ref{thm:the.analysis}.\ref{thm:liminf.T_n} becomes a direct
consequence of Theorems~\ref{thm:liminf} and
\ref{thm:the.analysis}.\ref{thm:T_n.demiclosed}.

\item
  Theorem~\ref{thm:the.analysis}.\ref{thm:nonempty.weak.cluster.points}
  guarantees that $\mathfrak{W}((u_n)_{n\in\Natural}) \neq
  \emptyset$. Fix arbitrarily a $u_*\in
  \mathfrak{W}((u_n)_{n\in\Natural})$. By definition, there exists a
  subsequence $N'\in \Natural_{\infty}^{\#}$ such that $u_n
  \xrightharpoonup{n\in N'} u_*$.

Recall Theorem~\ref{thm:the.analysis}.\ref{thm:fraction.goes.to.zero}
and easily verify that $u_n -
T_{\Theta_n}^{(\lambda_n)}(u_n)
\xrightarrow{n\in N'} 0$. This together with $u_n
\xrightharpoonup{n\in N'} u_*$ imply that
\begin{equation}
T_{\Theta_n}^{(\lambda_n)}(u_n)
\xrightharpoonup{n\in N'} u_*.\label{1st.eq.needed.for.demiclosed}
\end{equation}

Recall, now, Theorem~\ref{thm:the.analysis}.\ref{thm:T_n.demiclosed}
in order to obtain $(I-T)(T_{\Theta_n}^{(\lambda_n)}(u_n))
\xrightarrow{n\in N'} 0$. This result,
\eqref{1st.eq.needed.for.demiclosed}, and
Definition~\ref{def:demiclosed} lead us to $(I-T)(u_*)=0
\Leftrightarrow u_*\in\Fix(T)$. This establishes
Theorem~\ref{thm:the.analysis}.\ref{thm:weak.cluster.points.in.fixT}.

\item This is a direct consequence of
  Theorem~\ref{thm:the.analysis}.\ref{thm:weak.cluster.points.in.fixT}
  and the well-known fact that $\mathfrak{S}((u_n)_{n\in \Natural})
  \subset \mathfrak{W}((u_n)_{n\in \Natural})$.

\item It is easy to verify by
  Assumptions~\ref{main.assumptions}.\ref{ass:lambda.epsilon} and
  \ref{main.assumptions}.\ref{ass:eta} that 
\begin{equation*}
\frac{(2-\lambda_n)\eta_n}{2-\lambda_n(1-\eta_n)} =
\frac{(2-\lambda_n)\eta_n}{(2-\lambda_n) +\lambda_n\eta_n} \geq
\frac{\epsilon \check{\eta}}{2(1+\hat{\eta})} >0.
\end{equation*}
Using also \eqref{strongly.attracting.Omega_n}, we easily verify under
Assumption~\ref{main.assumptions}.\ref{ass:nonempty.intersection}
that
\begin{equation}
\forall n\in N, \forall v\in\Omega, \quad \frac{\epsilon
  \check{\eta}}{2(1+\hat{\eta})} \norm{u_n - 
  u_{n+1}}^2 \leq \norm{u_n - v}^2  - 
\norm{u_{n+1} - v}^2.\label{special.Fejer}
\end{equation}
The claim of
Theorem~\ref{thm:the.analysis}.\ref{thm:strong.convergence} is a
direct consequence of \eqref{special.Fejer},
Assumption~\ref{main.assumptions}.\ref{ass:ri.Omega.hyperplane}, and
Fact~\ref{fact:special.Fejer}.\qedhere

\end{enumerate}
\end{proof}

\section{Special Cases of the General Algorithm}\label{sec:examples}

\subsection{Exploring $(T_n)_{n\in\Natural}$.}\label{sec:apriori}

The available a-priori information about the model \eqref{model}
enters Algorithm~\ref{the.algo} through the sequence of mappings
$(T_n)_{n\in\Natural}$, i.e., implicitly via the sequence of sets
$(\Fix(T_n))_{n\in \Natural}$. Given that $n\in \Natural$ stands for
time, the sequence $(T_n)_{n\in\Natural}$ aims to capture the dynamic
nature of a-priori information, which is usually met in signal
processing and machine learning applications. For example, it is often
the case in adaptive signal processing to face a channel whose impulse
response changes slowly with time. Notice also here that the sequence
$(T_n)_{n\in\Natural}$ belongs to the rich family of strongly
attracting quasi-nonexpansive mappings. To demonstrate the versatility
offered by this class of mappings in the usage of the available
a-priori knowledge, examples of such mappings, mobilized extensively in
various contexts of optimization theory
\cite{bauschke.combettes.book}, are demonstrated in this section. More
specifically, in order to apply the proposed scheme to a real-world
problem, the following Example~\ref{def:l1.loss} considers a
non-smooth loss function which infuses sparsity information in
\eqref{model}. Such a loss function will be incorporated in
Algorithm~\ref{ex:subgrad.Thetan} to devise an algorithmic solution to
the online sparse system/signal recovery task of
Section~\ref{sec:application}.

\begin{example}[Resolvent]\label{ex:resolvent}
For a set-valued mapping $A:\hilbert \rightarrow 2^{\hilbert}$, its
\textit{graph} is defined as the set $\graph(A):=
\{(x,y)\in\hilbert\times\hilbert: y\in A(x)\}$. The mapping $A$ will
be called \textit{monotone} if $\forall (x_1,y_1), (x_2,y_2)\in
\graph(A)$, $\innprod{x_1-x_2}{y_1-y_2}\geq 0$
\cite{bauschke.combettes.book, Minty, BCR.resolvents,
  Rockafellar.Wets}. A monotone mapping $A$ will be called
\text{maximal} if no enlargement of its graph is possible without
destroying monotonicity, i.e., $\forall (x,y)\in
\hilbert\times\hilbert \setminus\graph(A)$, there exists a pair
$(x_0,y_0)\in\graph(A)$ such that $\innprod{x-x_0}{y-y_0}<0$
\cite{bauschke.combettes.book, Minty, BCR.resolvents,
  Rockafellar.Wets}. For example, the linear mapping induced by
any positive semi-definite matrix is maximal monotone \cite[Examples
  12.2 and 12.7]{Rockafellar.Wets}.

Now, given a maximal monotone mapping $A:\hilbert \rightarrow
2^{\hilbert}$, and a $\xi>0$, its \textit{resolvent}
$T^{(\xi)}:=(I+\xi A)^{-1}: \hilbert \rightarrow\hilbert$ is an
$1$-attracting nonexpansive mapping, where $(\cdot)^{-1}$ stands for
the inverse of a mapping. The fixed point set of $T^{(\xi)}$ becomes
$\Fix(T^{(\xi)})= \{x\in\hilbert: 0\in A(x)\}$. For example, in the
case of a positive semi-definite matrix, this fixed point set is
nothing but the null space of the matrix.
\end{example}

\begin{example}[Proximity mapping]\label{ex:proximal.mapping}
Given a lower semi-continuous function $\Phi: \hilbert \rightarrow
\Real$, the \textit{Moreau envelope of index $\gamma>0$} of $\Phi$ is the
function
\begin{equation}
\Phi^{(\gamma)}: \hilbert\rightarrow\Real: x\mapsto
\inf_{y\in\hilbert} \left(\Phi(y) + \frac{1}{2\gamma}\norm{x-y}^2
\right).\label{Moreau.envelope}
\end{equation}
Then, the \textit{proximity mapping} $T_{\gamma \Phi}$ is defined as
the mapping which maps to an $x\in\hilbert$ the \textit{unique}
minimizer of \eqref{Moreau.envelope}
\cite{Combettes.Pesquet.proximal.2010, Moreau.proximity.map,
  Combettes.Wajs}. It can be verified that the proximity mapping
$T_{\gamma \Phi}$ is $1$-attracting nonexpansive with fixed point set
$\Fix(T_{\gamma \Phi})= \{x\in\hilbert:
\Phi(x)=\inf_{y\in\hilbert}\Phi(y)\}$ \cite{Combettes.Wajs,
  Combettes.Pesquet.proximal.2010}.
\end{example}

\begin{example}[Inconsistent a-priori
    information]\label{ex:inconsistent} Assume that the available
  a-priori knowledge about our system is a gathering of several pieces
  of information which take the form of the following nonempty closed
  convex sets: $\Gamma$, $\{C_m\}_{m=1}^M$ in $\hilbert$, with $M\in
  \Naturalstar$. With $\Gamma$ we denote the information that our
  system should surely satisfy, called the \textit{absolute} or
  \textit{hard constraint}. Ideally, our solution set is $\Gamma \cap
  (\bigcap_{m=1}^M C_m)$. However, it is quite often the case that the
  available pieces of a-priori knowledge are inconsistent, i.e., the
  previous intersection is the empty set, e.g., \cite{sy.wideband}. To
  tackle such a problem, we define the following \textit{proximity
    function:} $\forall x\in\hilbert$, $p(x):= \sum_{m=1}^M \beta_m
  d^2(x,C_m)$, where $\{\beta_m\}_{m=1}^M$ are convex weights, i.e.,
  $\{\beta_m\}_{m=1}^M \subset (0,1]$, such that
$\sum_{m=1}^M\beta_m=1$. The proximity function is everywhere
Fr\'{e}chet differentiable, and its differential is the mapping $p':=
2\sum_{m=1}^M \beta_m(I-P_{C_m}): \hilbert\rightarrow
\hilbert$. Define, now, as our new solution set $\Xi:= \argmin\{p(x):
x\in\Gamma \}$. The non-emptiness of $\Xi$ is guaranteed if at least
one of $\{C_m\}_{m=1}^M$ or $\Gamma$ is bounded
\cite{Yamada.HSDM.2001}. In words, $\Xi$ is the set of all those
points in $\Gamma$ that least violate, in the sense of the previous
proximity function, the rest of the constraints
$\{C_m\}_{m=1}^M$. Under the previous setting, and $\forall
\lambda\in(0,2)$, the mapping $T_p:= P_{\Gamma}(I-\lambda p')$, is
$(1-\frac{\lambda}{2})$-attracting nonexpansive with fixed point set
$\Fix(T_p)=\Xi$ \cite{yukawa.splitting, CensorSplitFeasibility,
  CombettesHardConstr, YamadaetalQuadrOptim, Yamada.HSDM.2001}.
\end{example}

\begin{example}[The class $\mathfrak{T}$ of mappings
    \cite{BauschkeCombettesWeak2Strong}]\label{ex:cutter} For any
  $x,y\in\hilbert$, define the following set: $H(x,y):=
  \{v\in\hilbert: \innprod{x-y}{v-y}\leq 0\}$. In words, the set
  $H(x,y)$ is the closed halfspace onto which $y$ is the metric
  projection of $x$. Now, a mapping $T:\hilbert \rightarrow\hilbert$
  is said to belong to the class $\mathfrak{T}$ of mappings, if
  $\forall x\in\hilbert$, $\Fix(T) \subset H(x,T(x))$
  \cite{BauschkeCombettesWeak2Strong}. An equivalent description of
  the class $\mathfrak{T}$ is as follows: $T\in \mathfrak{T}$ iff $T$
  is firmly quasi-nonexpansive \cite[Proposition
    2.3]{BauschkeCombettesWeak2Strong}. Moreover, $\forall T\in
  \mathfrak{T}$, $\Fix(T)= \bigcap_{x\in\hilbert} H(x,T(x))$. For
  example, the subgradient projection mapping $T_{\Theta}$
  (Example~\ref{def:subgrad.projection}) belongs to this class
  \cite[Proposition 2.3]{BauschkeCombettesWeak2Strong}.
\end{example}

\begin{definition}[Sparsity-aware loss function]\label{def:l1.loss}
Henceforth, the notation $\overline{j_1,j_2}$, for any integers
$j_1\leq j_2$, will stand for $\{j_1, j_1+1, \ldots, j_2\}$. Assume
that $\hilbert:= \Real^L$, for some $L\in \Naturalstar$. We introduce,
here, the following sequence of convex, continuous, non-negative
functions $(\Phi_n: \hilbert \rightarrow
[0,\infty))_{n\in\Natural}$. Given a sequence of weight vectors
  $(\bm{w}_{n})_{n\in\Natural} \subset \Real^L$, with positive
  components, i.e., $w_{n,j}>0$, $\forall j\in\overline{1,L}$,
  $\forall n\in\Natural$, and a positive parameter $\rho > 0$, we
  define
\begin{equation}
\forall n\in\Natural, \forall \bm{x}\in\Real^L, \quad \Phi_n(\bm{x}) :=
\max\{0, \sum_{j=1}^L w_{n,j} |x_j|-\rho\}. \label{l1.function}
\end{equation}
It is clear that the $0$-th level set for each $\Phi_n$ is a
weighted $\ell_1$-ball, i.e.,
\begin{equation*}
\forall n\in\Natural, \quad \lev\Phi_n = B_{\ell_1}[\bm{w}_n,\rho] :=
\{\bm{x}\in\Real^L:\ \sum_{j=1}^L w_{n,j} |x_j| \leq \rho\}.
\end{equation*}
The fixed point set of the relaxed subgradient
projection mapping $T_{\Phi_n}^{(\nu_n)}$, $\nu_n\in (0,2)$, is the
weighted $\ell_1$-ball, i.e., $\Fix(T_{\Phi_n}^{(\nu_n)}) =
B_{\ell_1}[\bm{w}_n,\rho]$. The sequence $B_{\ell_1}[\bm{w}_n,\rho]$
has been very useful in building sparsity-aware online learning
methods in \cite{kst@tsp11, skt@icassp11, skt.eusipco11}. There, the metric
projection mapping $P_{B_{\ell_1}[\bm{w}_n,\rho]}$ was employed, whose computation
scales to the order of $\mathcal{O}(L\log_2L)$. 

Following a different path than \cite{kst@tsp11, skt@icassp11,
  skt.eusipco11}, the information carried by
$(B_{\ell_1}[\bm{w}_n,\rho])_{n\in\Natural}$ is viewed from an
alternative angle in this study: $\forall n\in \Natural$,
$B_{\ell_1}[\bm{w}_n,\rho]$ is not just a closed convex set, onto
which we project, but it is also the set of minimizers of the non-smooth
loss function $\Phi_n$. In order to minimize the non-smooth $\Phi_n$,
the subgradient information will be used. However, the employment of
such an information is not possible via \cite{YamadaAPSM,
  YamadaOguraAPSMNFAO, Kostas.APSM.NFAO}, since the subgradient
projection mapping (Definition~\ref{def:subgrad.projection}) belongs
to the class of strongly attracting quasi-nonexpansive mappings, which
is strictly larger than the class of strongly attracting nonexpansive
operators, utilized in \cite{Kostas.APSM.NFAO}.
\end{definition}

The set $B_{\ell_1}[\bm{w}_n,\rho]$ is a closed convex set, and its
metric projection mapping is given as follows. To save space, we give
here a short description. For the full discussion, the interested
reader can refer to \cite{kst@tsp11}.

\begin{fact}[Metric projection mapping onto the weighted $\ell_1$-ball
    \cite{kst@tsp11}]\label{fact:project.li.ball} Given
  $\bm{x}\in\Real^L\setminus B_{\ell_1}[\bm{w}_n,\rho]$, there exists
  an $l_* \in\overline{1,L}$, and a set of integers
  $\{l_j\}_{j\in\overline{l_*+1,L}} \subset \overline{l_*+1,L}$, such
  that the metric projection $P_{B_{\ell_1}[\bm{w}_n,\rho]}(\bm{x})$
  is given by a permutation on the components of the following vector
\begin{equation}
\left[x_1 - \frac{\sum_{i=1}^{l_*} w_{n,i}|x_i| -\rho}{\sum_{i=1}^{l_*}
  w_{n,i}^2} \sign(x_1) w_{n,1}, \ldots, x_{l_*} -
  \frac{\sum_{i=1}^{l_*} w_{n,i}|x_i| -\rho}{\sum_{i=1}^{l_*}
    w_{n,i}^2} \sign(x_{l_*}) w_{n,l_*}, 0,\ldots,
  0\right]^t, \label{project.onto.l1.ball}
\end{equation}
where
\begin{equation*}
\begin{cases}
|x_j| > \frac{\sum_{i=1}^{l_*} w_{n,i}|x_i| -\rho}{\sum_{i=1}^{l_*}
  w_{n,i}^2} w_{n,j}, & \forall j\in\overline{1,l_*},\\
|x_j| \leq \frac{\sum_{i=1}^{l_j} w_{n,i}|x_i| -\rho}{\sum_{i=1}^{l_j}
  w_{n,i}^2} w_{n,j}, & \forall j\in\overline{l_*+1,L}.
\end{cases}
\end{equation*}
Without any loss of generality, we assume that
$P_{B_{\ell_1}[\bm{w}_n,\rho]}(\bm{x})$ is given by
\eqref{project.onto.l1.ball} in the sequel.
\end{fact}

Regarding Definition~\ref{def:l1.loss}, consider the following assumptions.

\begin{assumption}\label{ass:l1.loss}\mbox{}
\begin{enumerate}[1.]

\item \label{ass:weight.vectors}
The sequence of weight vectors $(\bm{w}_n)_{n\in\Natural}$
is constructed such that $\forall n\in\Natural$, $\forall
j\in\overline{1,L}$, $w_{n,j}\in [\check{\epsilon}, \hat{\epsilon}]$,
for some $\check{\epsilon}, \hat{\epsilon}>0$.

\item\label{ass:nu} Given the sequence of relaxed subgradient
  projection mappings $(T_{\Phi_n}^{(\nu_n)})_{n\in \Natural}$, with
  respect to the sequence $(\Phi_n)_{n\in \Natural}$ in
  Definition~\ref{def:l1.loss}, there exists $\epsilon'\in (0,1]$ such
that $\forall n\in\Natural$, $\nu_n\in[\epsilon', 2-\epsilon']$.

\end{enumerate}
\end{assumption}

\begin{table}[t]
\center
\begin{tabular}{|c|c|c|}
\hline $\bm{x}$ & & $\partial\Phi_n(\bm{x})$ \\ \hline\hline
$\sum_{j=1}^n w_{n,j} |x_j| <\rho$, & & $\{\bm{0}\}$.\\ \hline
$\sum_{j=1}^n w_{n,j} |x_j| >\rho$, & $\mathfrak{J}_{\bm{x}}
=\emptyset$, &
\begin{minipage}{10cm}\center
\vspace{5pt}
$\left\{\left[\begin{smallmatrix}
w_{n,1}\sign(x_1)\\
\vdots\\ w_{n,L}\sign(x_L)
\end{smallmatrix}\right]\right\}$.
\vspace{5pt}\end{minipage}\\ \hline
$\sum_{j=1}^n w_{n,j} |x_j| >\rho$, & $\mathfrak{J}_{\bm{x}}
 \neq \emptyset$, & 
\begin{minipage}{10cm}\center
\vspace{5pt}
$\conv \left\{\bm{u}_1,\ldots,
 \bm{u}_{2^{\tau}}\right\}$, \\ 
where the vectors $\bm{u}_k$,
 $\forall k\in\overline{1,2^{\tau}}$, are given by\\
$u_{k,j}:=\begin{cases}w_{n,j}\sign(x_j), & \text{if}\ j\notin
\mathfrak{J}_{\bm{x}},\\ \pm w_{n,j}, &
\text{if}\ j\in \mathfrak{J}_{\bm{x}}.\end{cases}$\vspace{5pt}
\end{minipage}
\\ \hline
$\sum_{j=1}^n w_{n,j} |x_j| =\rho$, & $\mathfrak{J}_{\bm{x}} =
 \emptyset$,
 & \begin{minipage}{10cm}\center
\vspace{5pt}
$\conv\left\{\bm{0}, \left[\begin{smallmatrix}
w_{n,1}\sign(x_1)\\
\vdots\\
w_{n,L}\sign(x_L)\end{smallmatrix}\right]\right\}$.\vspace{5pt}
\end{minipage}\\ \hline
$\sum_{j=1}^n w_{n,j} |x_j| =\rho$, & $\mathfrak{J}_{\bm{x}}\neq
 \emptyset$, & $\conv \{\bm{0},\bm{u}_1,\ldots,
 \bm{u}_{2^{\tau}}\}$.\\ \hline
\end{tabular}\vspace*{10pt}
\caption{Here,
  $\mathfrak{J}_{\bm{x}}:=\{j \in\overline{1,L}:\ x_j=0\}$, and $\tau$
  stands for the cardinality of $\mathfrak{J}_{\bm{x}}$, whenever
  $\mathfrak{J}_{\bm{x}}\neq \emptyset$. The $\conv$ symbol stands for
  the convex hull of a set.}\label{table:subgradient.linear.l1norm}
\end{table}

\begin{lemma}\label{lem:l1.function}
The following properties hold true.
\begin{enumerate}[1.]

\item\label{lem:l1.subdiff} The subdifferentials of the loss functions
  $(\Phi_n)_{n\in\Natural}$, defined in \eqref{l1.function}, are given
  in Table~\ref{table:subgradient.linear.l1norm}.

\item\label{lem:l1.bounded.subdiff} Let
  Assumption~\ref{ass:l1.loss}.\ref{ass:weight.vectors} hold
  true. Then, $\forall \bm{x}\in \Real^L$, $(\Phi_n'(\bm{x}))_{n\in
    \Natural}$ is bounded.

\item\label{lem:l1.interior} Let
  Assumption~\ref{ass:l1.loss}.\ref{ass:weight.vectors} hold
  true. Then, $\interior(\bigcap_{n\in\Natural}
  B_{\ell_1}[\bm{w}_n,\rho])\neq \emptyset$.

\item\label{lem:l1.liminf.ass} Let
  Assumptions~\ref{ass:l1.loss}.\ref{ass:weight.vectors} and
  \ref{ass:l1.loss}.\ref{ass:nu} hold true. Then, the sequence of
  relaxed subgradient projection mappings
  $(T_{\Phi_n}^{(\nu_n)})_{n\in\Natural}$ satisfies
  Assumption~\ref{ass:liminf.Fn}.

\end{enumerate}
\end{lemma}

\begin{proof}
\begin{enumerate}[1.]

\item To save space, the calculation of the subdifferentials in
  Table~\ref{table:subgradient.linear.l1norm} is omitted. These
  results can be reproduced by using standard arguments of convex
  analysis, e.g., \cite[Thm.\ 25.6]{Rockafellar}.

\item Lemma~\ref{lem:l1.function}.\ref{lem:l1.bounded.subdiff} can be
  easily established by
  Assumption~\ref{ass:l1.loss}.\ref{ass:weight.vectors} and
  Table~\ref{table:subgradient.linear.l1norm}.

\item Choose any $\bm{x}\in
  B(\bm{0},\frac{\rho}{L\hat{\epsilon}})$. Then, $\forall
  j\in\overline{1,L}$, $|x_j| \leq
  \frac{\rho}{L\hat{\epsilon}}$. Moreover, $\sum_{j=1}^L w_{n,j}
  |x_j|\leq \sum_{j=1}^L \hat{\epsilon} \frac{\rho}{L\hat{\epsilon}} =
  \rho$. Hence, $B(\bm{0},\frac{\rho}{L\hat{\epsilon}}) \subset
  B_{\ell_1}[\bm{w}_n,\rho]$, $\forall n\in\Natural$. This clearly
  suggests that $\bm{0}\in \interior(\bigcap_{n\in\Natural}
  B_{\ell_1}[\bm{w}_n,\rho])$, which establishes
  Lemma~\ref{lem:l1.function}.\ref{lem:l1.interior}.

\item First, notice that $\forall n\in\Natural$,
  $\Fix(T_{\Phi_n}^{(\nu_n)}) = B_{\ell_1}[\bm{w}_n,\rho]$. Now,
according to Assumption~\ref{ass:liminf.Fn}, fix arbitrarily a
subsequence $N\in \Natural_{\infty}^{\#}$, a sequence
$(\bm{x}_n)_{n\in N}\subset\Real^L$, and a $\gamma>0$ such that
$\forall n\in N$, $d(\bm{x}_n, B_{\ell_1}[\bm{w}_n,\rho]))\geq
\gamma$. Notice by Fact~\ref{fact:project.li.ball} the following:
$\forall n\in N$,
\begin{align*}
\gamma^2 & \leq d^2(\bm{x}_n, B_{\ell_1}[\bm{w}_n,\rho]) = \norm{\bm{x}_n -
  P_{B_{\ell_1}[\bm{w}_n,\rho]}(\bm{x}_n)}^2 \\
& = \sum_{j=1}^{l_*} \frac{\left(\sum_{i=1}^{l_*} w_{n,i}|x_{n,i}| -
  \rho \right)^2}{\left(\sum_{i=1}^{l_*} w_{n,i}^2\right)^2} w_{n,j}^2
+ \sum_{j=l_*+1}^L x_{n,j}^2\\
& \leq \sum_{j=1}^{l_*} \frac{\left(\sum_{i=1}^{l_*} w_{n,i}|x_{n,i}| -
  \rho \right)^2}{\left(\sum_{i=1}^{l_*} w_{n,i}^2\right)^2} w_{n,j}^2
+ \sum_{j=l_*+1}^{L} \frac{\left(\sum_{i=1}^{l_j} w_{n,i}|x_{n,i}| -
  \rho \right)^2}{\left(\sum_{i=1}^{l_j} w_{n,i}^2\right)^2}
w_{n,j}^2\\
& \leq \frac{\left(\sum_{i=1}^{L} w_{n,i}|x_{n,i}| -
  \rho \right)^2}{\left(\sum_{i=1}^{l_*} w_{n,i}^2\right)^2}
\sum_{j=1}^L w_{n,j}^2,
\end{align*}
which, in turn, results into
\begin{equation*}
\Phi_{n}^2(\bm{x}_n)  = \left(\sum_{i=1}^{L} w_{n,i}|x_{n,i}| -
  \rho \right)^2 \geq \gamma^2 \frac{\left(\sum_{i=1}^{l_*}
    w_{n,i}^2\right)^2}{\sum_{j=1}^L w_{n,j}^2} \geq \gamma^2
  \frac{\check{\epsilon}^4}{L\hat{\epsilon}^2} =: \delta'^2>0,
  \quad \forall n\in N. 
\end{equation*}
Notice, also, by Example~\ref{def:subgrad.projection} and
Lemma~\ref{lem:l1.function}.\ref{lem:l1.bounded.subdiff} that $\forall
n\in N$,
\begin{align*}
\norm{(I-T_{\Phi_{n}}^{(\nu_n)})(\bm{x}_n)} = \nu_n
\frac{\Phi_n(\bm{x}_n)}{\norm{\Phi_n'(\bm{x}_n)}} \geq \check{\epsilon}
\frac{\delta'}{D} >0,
\end{align*}
which clearly suggests that $\exists \delta>0$ such that
$\liminf_{n\in N} \norm{(I-T_{\Phi_n}^{(\nu_n)})(\bm{x}_n)} 
\geq \delta$. This establishes
Lemma~\ref{lem:l1.function}.\ref{lem:l1.liminf.ass}.\qedhere

\end{enumerate}
\end{proof}

\subsection{Exploring $(\Theta_n)_{n\in\Natural}$.}\label{sec:loss.functions}

In this section, the metric distance function to closed convex sets
will be used in order to define a sequence of loss functions
$(\Theta_n)_{n\in\Natural}$. Such sequences have already found
numerous applications in online signal processing and machine learning
tasks \cite{KostasClassifyTSP, ksi.Beam.TSP, tsy.sp.magazine,
  yukawa.splitting, YKostasY}, under the light, however, of the
predecessors \cite{YamadaAPSM, YamadaOguraAPSMNFAO, Kostas.APSM.NFAO}
of the present framework. In this section, this specific sequence
$(\Theta_n)_{n\in\Natural}$ will be blended with the more general
class of strongly attracting quasi-nonexpansive mappings in order to
construct Algorithm~\ref{ex:subgrad.Thetan}. Given the wide
applicability of the techniques in \cite{YamadaAPSM,
  YamadaOguraAPSMNFAO, Kostas.APSM.NFAO}, it is natural to anticipate
an even larger span of usage for
Algorithm~\ref{ex:subgrad.Thetan}. Such a potential will be
demonstrated in Section~\ref{sec:application}, where
Algorithm~\ref{ex:subgrad.Thetan} is applied to the online
sparse system/signal recovery task.

\begin{definition}\label{ex:extrapolated.distance.functions}
Assume a sequence of nonempty closed convex sets
$(S_n)_{n\in\Natural}$. Given a user-defined $q\in\Naturalstar$,
let the following index set
\begin{equation*}
\mathcal{J}_n := \overline{\max\{0,n-q+1\},n}, \quad\forall n\in\Natural.
\end{equation*}
Notice that the sequence $(\mathcal{J}_n)_{n\in\Natural}$
depicts a sliding window on the set $\Natural$, of length at most
$q$.

Let us introduce a sequence of convex functions $(\Theta_n:
\hilbert\rightarrow [0,\infty))_{n\in\Natural}$ inductively. For every
$n\in\Natural$, and given a $u_n\in\hilbert$, define the following
\textit{active index set:}
\begin{equation*}
\mathcal{I}_n := \{i\in\mathcal{J}_n:\ u_n\notin S_i\}. 
\end{equation*}
This set identifies those closed convex sets
$\{S_i\}_{i\in\mathcal{I}_n}$, out of $\{S_j\}_{j\in\mathcal{J}_n}$,
which add on new ``information'' to our learning process. The sets
with indexes $\{j\in\mathcal{J}_n: u_n\in S_j\}$ will not be processed
at the time instant $n$.

In the case where $\mathcal{I}_n\neq\emptyset$, we introduce the set
of weights $\{\omega_i^{(n)}\}_{i\in \mathcal{I}_n}\subset(0,1]$, such
that $\sum_{i\in\mathcal{I}_n} \omega_i^{(n)} = 1$. Define, now, the
convex function:
\begin{equation}
\forall x\in\hilbert, \quad \Theta_n(x) := \begin{cases}
 \sum_{i\in\mathcal{I}_n} \frac{\omega_i^{(n)}
d(u_n,S_i)}{L_n} d(x,S_i), &
\text{if}\ \mathcal{I}_n \neq \emptyset,\\
0, & \text{if}\ \mathcal{I}_n = \emptyset,
\end{cases}\label{special.Theta_n}
\end{equation}
where $L_n := \sum_{i\in\mathcal{I}_n} \omega_i^{(n)} d(u_n,S_i)$. We
define $L_n:=0$ for all those $n\in \Natural$ such that
$\mathcal{I}_n=\emptyset$.
\end{definition}

\begin{lemma}\label{lem:properties.example.theta}
The following properties hold true for the sequence of functions
$(\Theta_n)_{n\in\Natural}$ given in \eqref{special.Theta_n}.

\begin{enumerate}[1.]

\item\label{lem:Ln.positive} For every $n\in\Natural$, such that
  $\mathcal{I}_n \neq \emptyset$, we have $L_n>0$.

\item\label{lem:level.sets.example.theta} For every $n\in\Natural$,
  $\lev\Theta_n = \bigcap_{i\in\mathcal{I}_n} S_i$, where we define
  $\bigcap_{i\in\emptyset} S_i:= \hilbert$, to cover also the case
  where $\mathcal{I}_n = \emptyset$.

\item\label{lem:bounded.subgrads} The collection of all the
  subgradients of $(\Theta_n)_{n\in\Natural}$ is bounded, i.e.,
  $\forall n\in\Natural$, $\forall x\in\hilbert$,
  $\norm{\Theta_n'(x)}\leq 1$.

\item\label{lem:subgradient.special.Theta.un} For any
  $n\in\Natural$,
\begin{equation*} 
\Theta_n'(u_n) = \begin{cases} \frac{1}{L_n}
  \sum_{i\in\mathcal{I}_n} \omega_i^{(n)} (u_n - P_{S_i}(u_n)), &
  \mathcal{I}_n\neq \emptyset,\\ 0, & \mathcal{I}_n=\emptyset.
\end{cases}
\end{equation*}

\end{enumerate}
\end{lemma}

\begin{proof}
\begin{enumerate}[1.]

\item Fix arbitrarily an $n\in\Natural$ such that $\mathcal{I}_n\neq
  \emptyset$. By the definition of $\mathcal{I}_n$, $\forall
  i\in\mathcal{I}_n$, $d(u_n,S_i)>0$. Since, also,
  $\omega_i^{(n)}\in(0,1]$, $\forall i\in\mathcal{I}_n$, it is clear
    by the definition of $L_n$ that
    Lemma~\ref{lem:properties.example.theta}.\ref{lem:Ln.positive}
    holds true.

\item Fix arbitrarily an $n\in\Natural$. Assume, first, that
  $\mathcal{I}_n=\emptyset$. By \eqref{special.Theta_n}, it is clear
  that $\lev\Theta_n = \hilbert =: \bigcap_{i\in\emptyset} S_i$.

  Assume, now, that $\mathcal{I}_n\neq \emptyset$, and that
  $\bigcap_{i\in\mathcal{I}_n} S_i \neq \emptyset$. It is clear by
  \eqref{special.Theta_n} that $\bigcap_{i\in\mathcal{I}_n} S_i
  \subset \lev\Theta_n$. Assume, now, an $x\notin
  \bigcap_{i\in\mathcal{I}_n} S_i$, or equivalently, $\exists
  i_0\in\mathcal{I}_n$ such that $d(x,S_{i_0})>0$. Then, one can
  easily verify that $\Theta_n(x)\geq \frac{\omega_{i_0}^{(n)}
    d(u_n,S_{i_0})}{L_n} d(x,S_{i_0})>0$. In other words, $x\notin
  \lev\Theta_n$, and finally $\lev\Theta_n \subset
  \bigcap_{i\in\mathcal{I}_n} S_i$. Notice that the previous arguments
  hold true also in the case where $\bigcap_{i\in\mathcal{I}_n} S_i=
  \emptyset$. This establishes
  Lemma~\ref{lem:properties.example.theta}.\ref{lem:level.sets.example.theta}.

\item Fix arbitrarily an $n\in\Natural$. By \eqref{special.Theta_n},
  basic calculus on subdifferentials \cite{EkelandTemam,
    Rockafellar.Wets} suggests that
\begin{equation*}
\forall x\in\hilbert, \quad \partial\Theta_n(x) = \begin{cases}
\sum_{i\in\mathcal{I}_n} \frac{\omega_i^{(n)} d(u_n,S_i)}{L_n}
\partial d(x,S_i), & \text{if}\ \mathcal{I}_n\neq \emptyset,\\
\{0\}, & \text{if}\ \mathcal{I}_n = \emptyset.
\end{cases}
\end{equation*}
From now and on, we deal only with the case where
$\mathcal{I}_n\neq\emptyset$, since the previous equation
clearly suggests that
Lemma~\ref{lem:properties.example.theta}.\ref{lem:bounded.subgrads}
holds trivially in the case of $\mathcal{I}_n=\emptyset$.

By Example~\ref{ex:partial.distance}, the subgradient $\Theta_n'(x)$
takes the following form:
\begin{align}
\forall x\in\hilbert, \quad \Theta_n'(x) & = 
\sum_{i\in\mathcal{I}_n:\ x\notin S_i} \frac{\omega_i^{(n)}
  d(u_n,S_i)}{L_n} d'(x,S_i) + \sum_{i\in\mathcal{I}_n:\ x\in S_i}
\frac{\omega_i^{(n)} d(u_n,S_i)}{L_n} d'(x,S_i) \nonumber\\
& = \sum_{i\in\mathcal{I}_n:\ x\notin S_i} \frac{\omega_i^{(n)}
  d(u_n,S_i)}{L_n} \frac{x-P_{S_i}(x)}{d(x,S_i)} + \sum_{i\in\mathcal{I}_n:\ x\in S_i}
\frac{\omega_i^{(n)} d(u_n,S_i)}{L_n}
d'(x,S_i). \label{subgradient.special.Theta} 
\end{align}
Hence,
\begin{align*}
\forall x\in\hilbert, \quad \norm{\Theta_n'(x)} & \leq
\sum_{i\in\mathcal{I}_n:\ x\notin 
 S_i} \frac{\omega_i^{(n)} d(u_n,S_i)}{L_n}
 \frac{\norm{x- P_{S_i}(x)}}{d(x,S_i)} + \sum_{i\in\mathcal{I}_n:\ x\in
 S_i} \frac{\omega_i^{(n)} d(u_n,S_i)}{L_n}\cdot 1\\
& = \sum_{i\in\mathcal{I}_n:\ x\notin
 S_i} \frac{\omega_i^{(n)} d(u_n,S_i)}{L_n} + \sum_{i \in\mathcal{I}_n:\ x\in
 S_i} \frac{\omega_i^{(n)} d(u_n,S_i)}{L_n} = 1.
\end{align*}
This establishes
Lemma~\ref{lem:properties.example.theta}.\ref{lem:bounded.subgrads}.

\item
  Lemma~\ref{lem:properties.example.theta}.\ref{lem:subgradient.special.Theta.un}
  is an immediate consequence of
  \eqref{subgradient.special.Theta}.\qedhere

\end{enumerate}
\end{proof}

\begin{algo}\label{ex:subgrad.Thetan}
Assume a sequence of nonempty closed convex sets $(S_n)_{n\in\Natural}
\subset\hilbert$. Moreover, consider a sequence of convex continuous
functions $(\Phi_n:\hilbert \rightarrow \Real)_{n\in\Natural}$, such
that $\lev\Phi_n \neq \emptyset$, $\forall n\in\Natural$. Associated
to each $\Phi_n$ is the relaxed subgradient projection mapping
$T_{\Phi_n}^{(\nu_n)}$ (see Definition~\ref{def:subgrad.projection}),
where $\nu_n\in (0,2)$, $\forall n\in\Natural$.

For an arbitrarily chosen $u_0\in\hilbert$, form the following
sequence:
\begin{equation*}
\forall n\in\Natural, \quad u_{n+1} := \begin{cases}
T_{\Phi_n}^{(\nu_n)}\left(u_n - \lambda_n
\frac{\Theta_n(u_n)}{\norm{\Theta_n'(u_n)}^2}
\Theta_n'(u_n) \right), & \text{if}\ \Theta_n'(u_n) \neq
0,\\
T_{\Phi_n}^{(\nu_n)}(u_n), & \text{if}\ \Theta_n'(u_n) = 0,
\end{cases}
\end{equation*}
where the sequence of functions $(\Theta_n)_{n\in\Natural}$ is
given in Definition~\ref{ex:extrapolated.distance.functions},
$\Theta_n'(u_n)$ is any subgradient of $\Theta_n$ at
$u_n$, and $\lambda_n\in(0,2)$, $\forall n\in\Natural$.

Lemma~\ref{lem:properties.example.theta}.\ref{lem:subgradient.special.Theta.un}
and some elementary algebra lead to the following equivalent
formulation of the previous recursion:
\begin{equation}
\forall n\in\Natural, \quad u_{n+1} = T_{\Phi_n}^{(\nu_n)}\left( u_n
+ \mu_n \left( \sum_{i\in\mathcal{I}_n} \omega_i^{(n)} P_{S_i}(u_n) -
u_n \right)\right), \label{equiv.form.example}
\end{equation}
where $\mu_n:= \lambda_n \mathcal{M}_n$, and 
\begin{equation}
\mathcal{M}_n:= \begin{cases}
\frac{\sum_{i\in\mathcal{I}_n} \omega_i^{(n)} d^2(u_n,S_i)}
     {\norm{\sum_{i\in\mathcal{I}_n} \omega_i^{(n)} (u_n
         -P_{S_i}(u_n))}^2}, & \text{if}\ \sum_{i\in\mathcal{I}_n} \omega_i^{(n)}
     (u_n -P_{S_i}(u_n)) \neq 0,\\
1, & \text{otherwise}.
\end{cases}\label{Mu_n}
\end{equation}

To avoid any ambiguity in the case where $\mathcal{I}_n = \emptyset$,
we define in \eqref{equiv.form.example} and \eqref{Mu_n}:
$\sum_{i\in\emptyset} \omega_i^{(n)} (P_{S_i}(u_n)-u_n) :=
\sum_{i\in\emptyset} \omega_i^{(n)} P_{S_i}(u_n) -u _n := 0$. Notice
also by the convexity of $\norm{\cdot}^2$ that $\mathcal{M}_n \geq 1$,
and that since $\lambda_n\in (0,2)$, we obtain
$\mu_n\in(0,2\mathcal{M}_n)$, i.e., the extrapolation parameter
$\mu_n$ is able to take values greater than or equal to $2$, $\forall
n\in\Natural$.
\end{algo}

It is needless to say that the results presented in
Theorem~\ref{thm:the.analysis} hold true also for
Algorithm~\ref{ex:subgrad.Thetan}. Nevertheless, one can establish
additional properties for Algorithm~\ref{ex:subgrad.Thetan}, based on
the following assumptions.

\begin{assumption}\label{ass:main.example} Regarding
  Definition~\ref{ex:extrapolated.distance.functions} and
  Algorithm~\ref{ex:subgrad.Thetan}, assume the following.
\begin{enumerate}[1.]

\item\label{ass:omega} Let $\check{\omega} :=
  \inf\{\omega_i^{(n)}:\ i\in\mathcal{I}_n \neq\emptyset,
  n\in\Natural\}>0$.

\item\label{ass:Phi.bounded.subgrads} The sequences
  $(\Phi_n'(u_n))_{n\in\Natural}$ and
  $(\Phi_n'(T_{\Theta_n}^{(\lambda_n)}(u_n)))_{n\in\Natural}$ are
  bounded, i.e., there exists a $D >0$ such that $\forall
  n\in\Natural$, $\max\left\{\norm{\Phi_n'(u_n)},
  \norm{\Phi_n'(T_{\Theta_n}^{(\lambda_n)}(u_n))}\right\} \leq D$.

\item\label{ass:nonnegative.Phin} $\forall n\in\Natural$,
  $\Phi_n: \hilbert \rightarrow [0,\infty)$. See, for example,
  Definition~\ref{def:l1.loss}.

\end{enumerate}
\end{assumption}

\begin{theorem}\label{thm:subgrad.Thetan} The following statements are
  valid for Algorithm~\ref{ex:subgrad.Thetan}.

\begin{enumerate}[1.]

\item\label{thm:subgrad.Thetan.bound.Ln} Let
  Assumption~\ref{main.assumptions}.\ref{ass:nonempty.intersection}
  hold true. Then, there exists a $D> 0$ such that $\forall
  n\in\Natural$, $L_n\leq D$.

\item\label{thm:subgrad.Thetan.distances.go.to.zero} Let
  Assumptions~\ref{main.assumptions}.\ref{ass:nonempty.intersection},
  \ref{main.assumptions}.\ref{ass:lambda.epsilon}, and
  \ref{ass:main.example}.\ref{ass:omega} hold true. Then,
  $\lim_{n\rightarrow\infty} \max\{d(u_n,S_j):\ j\in\mathcal{J}_n\} =
  0$.

\item\label{thm:subgrad.Thetan.liminf} If
  Assumptions~\ref{main.assumptions}.\ref{ass:nonempty.intersection},
  \ref{main.assumptions}.\ref{ass:lambda.epsilon},
  \ref{main.assumptions}.\ref{ass:strong.clusters}, and
  \ref{ass:main.example}.\ref{ass:omega} hold true, then
  $\mathfrak{S}((u_n)_{n\in\Natural}) \subset
  \limsup_{n\rightarrow\infty} S_n$. Moreover, if there exists a
  $u_*\in\hilbert$ such that $\lim_{n\rightarrow\infty} u_n=u_*$,
  i.e., $\mathfrak{S}((u_n)_{n\in\Natural}) = \{u_*\}$, then $u_* \in
  \liminf_{n\rightarrow\infty} S_n$.

\item\label{thm:Phin.go.to.zero} Let
  Assumptions~\ref{main.assumptions}.\ref{ass:nonempty.intersection},
  \ref{main.assumptions}.\ref{ass:lambda.epsilon},
  \ref{ass:l1.loss}.\ref{ass:nu}, and
  \ref{ass:main.example}.\ref{ass:Phi.bounded.subgrads} hold
  true. Then, $\limsup_{n\rightarrow\infty}\Phi_n(u_n) \leq 0$. If, in
  addition,
  Assumption~\ref{ass:main.example}.\ref{ass:nonnegative.Phin} holds
  true, then $\lim_{n\rightarrow\infty}\Phi_n(u_n) = 0$.

\item\label{thm:sparsity.example} The following result applies to the
  next section where a system/signal recovery task is
  considered. Assume Algorithm~\ref{ex:subgrad.Thetan} for the case
  where $\hilbert:= \Real^L$, $L\in \Naturalstar$, equipped with the
  standard vector inner product. Assume, also, that the sequence of
  functions $(\Phi_n)_{n\in\Natural}$ is given by
  Definition~\ref{def:l1.loss}. Let
  Assumptions~\ref{main.assumptions}.\ref{ass:nonempty.intersection},
  \ref{main.assumptions}.\ref{ass:lambda.epsilon},
  \ref{ass:l1.loss}.\ref{ass:weight.vectors} and
  \ref{ass:l1.loss}.\ref{ass:nu} hold true. Then,
  $\mathfrak{S}((\bm{u}_n)_{n\in\Natural}) \subset
  \limsup_{n\rightarrow\infty} B_{\ell_1}[\bm{w}_n,\rho]$. If there
  exists a $\bm{u}_*$ such that $\lim_{n\rightarrow\infty}\bm{u}_n =
  \bm{u}_*$, then $\bm{u}_*\in \liminf_{n\rightarrow\infty}
  B_{\ell_1}[\bm{w}_n,\rho]$.

\end{enumerate}
\end{theorem}

\begin{proof}
\begin{enumerate}[1.]

\item Notice, that
$\forall n\in N$, $\forall i\in\mathcal{I}_n$, $\forall
  v\in\Omega$,
\begin{align*}
d(u_n,S_i) & = \norm{u_n -
 P_{S_i}(u_n)} \leq \norm{u_n - v} + \norm{v -
 P_{S_i}(u_n)} \leq 2 \norm{u_n - v} \leq 2 \norm{u_{n_0} - v},
\end{align*}
where $n_0:= \min N$, the second inequality follows from
Example~\ref{ex:relaxed.projection}, and the third one from
\eqref{monotonicity.Omega}. Now, by the definition of $L_n$, $\forall
n\in N$,
\begin{equation*}
L_n = \sum_{i\in\mathcal{I}_n} \omega_i^{(n)}
d(u_n,S_i) \leq 2 \sum_{i\in\mathcal{I}_n}
\omega_i^{(n)} \norm{u_{n_0} - v} = 2 \norm{u_{n_0} - v}. 
\end{equation*}
Choose, now, any $D > \max\{2\norm{u_{n_0} - v}, L_0, \ldots, L_{n_0
  -1}\}$, and notice that for such a $D$ the claim holds true.

\item Recall, here, by
  Definition~\ref{ex:extrapolated.distance.functions}, that if $u_n$
  is such that $\mathcal{I}_n = \emptyset$, then $d(u_n, S_j)=0$,
  $\forall j\in\mathcal{J}_n$. Obviously, this is equivalent to
  $\max\{d(u_n,S_j):\ j\in \mathcal{J}_n\} = 0$.

  Hence, we deal only with the case of $\mathcal{I}_n \neq
  \emptyset$. For this case, we observe by \eqref{special.Theta_n}
  that
\begin{align}
\Theta_n(u_n)  & = \sum_{i\in\mathcal{I}_n} \frac{\omega_i^{(n)}
d^2(u_n,S_i)}{L_n} \geq \sum_{i\in\mathcal{I}_n} \frac{\omega_i^{(n)}
d^2(u_n,S_i)}{D} \nonumber \\ 
&\geq \frac{\check{\omega}}{D} \sum_{i\in\mathcal{I}_n}
d^2(u_n,S_i) \geq \frac{\check{\omega}}{D}
\max\{d^2(u_n,S_i):\ i\in
\mathcal{I}_n\}, \label{subgrad.Thetan.an.inequality} 
\end{align}
where the existence of $D>0$ is guaranteed by
Theorem~\ref{thm:subgrad.Thetan}.\ref{thm:subgrad.Thetan.bound.Ln}. 

In order to establish
Theorem~\ref{thm:the.analysis}.\ref{asymptotic.optimality}, i.e.,
$\lim_{n\rightarrow\infty}\Theta_n(u_n)=0$, we have used
Assumption~\ref{main.assumptions}.\ref{ass:bounded.subgradients},
which imposes a bound on the sequence of subgradients
$(\Theta_n'(u_n))_{n\in\Natural}$. However, for the case at hand,
Lemma~\ref{lem:properties.example.theta}.\ref{lem:bounded.subgrads}
clearly suggests that boundedness holds true by default, that
Assumption~\ref{main.assumptions}.\ref{ass:bounded.subgradients} is
not necessary here, and that
Assumptions~\ref{main.assumptions}.\ref{ass:nonempty.intersection},
\ref{main.assumptions}.\ref{ass:lambda.epsilon} are sufficient for
establishing $\lim_{n\rightarrow\infty}\Theta_n(u_n)=0$. Having this
result hold true, apply $\lim_{n\rightarrow\infty}$ on both sides of
\eqref{subgrad.Thetan.an.inequality} to obtain
$\lim_{n\rightarrow\infty} \max\{d(u_n,S_i):\ i\in \mathcal{I}_n\} =
0$.

Recall, now, by the definition of $\mathcal{I}_n$, in
Definition~\ref{ex:extrapolated.distance.functions}, that $\forall j\in
\mathcal{J}_n \setminus \mathcal{I}_n$, $u_n\in S_j \Leftrightarrow
d(u_n,S_j)=0$. This clearly implies that $\max\{d(u_n,S_i):\ i\in
\mathcal{I}_n\} = \max\{d(u_n,S_j):\ j\in \mathcal{J}_n\}$. This
equality and the previously obtained result $\lim_{n\rightarrow\infty}
\max\{d(u_n,S_i):\ i\in \mathcal{I}_n\} = 0$ establish
Theorem~\ref{thm:subgrad.Thetan}.\ref{thm:subgrad.Thetan.distances.go.to.zero}.

\item We have already seen in
  Theorem~\ref{thm:subgrad.Thetan}.\ref{thm:subgrad.Thetan.distances.go.to.zero}
  that $\lim_{n\rightarrow\infty} \max\{d(u_n,S_j):\ j\in
  \mathcal{J}_n\}=0$. Since, by definition, $n\in\mathcal{J}_n$,
  $\forall n\in\Natural$, the previous result implies that
  $\lim_{n\rightarrow\infty} d(u_n, S_n)= \lim_{n\rightarrow\infty}
  \norm{(I-P_{S_n})(u_n)} = 0$. Having these in mind,
  Theorem~\ref{thm:subgrad.Thetan}.\ref{thm:subgrad.Thetan.liminf}
  becomes a direct consequence of Theorem~\ref{thm:liminf} and
  Example~\ref{ex:distance.function}.

\item Here, we will utilize
  Theorems~\ref{thm:the.analysis}.\ref{thm:fraction.goes.to.zero} and
  \ref{thm:the.analysis}.\ref{thm:T_n.demiclosed}. To this end, notice
  that regarding the sequence of mappings
  $(T_{\Phi_n}^{(\nu_n)})_{n\in\Natural}$,
  Assumption~\ref{main.assumptions}.\ref{ass:eta} is satisfied here;
  indeed, notice that $\forall n\in\Natural$, $\frac{\epsilon'}{2}
  \leq \frac{2-\nu_n}{\nu_n}\leq \frac{2}{\epsilon'}$.

Now, Definition~\ref{def:subgradient} suggests that $\forall
n\in\Natural$, $\innprod{T_{\Theta_n}^{(\lambda_n)}(u_n) -
  u_n}{\Phi'_n(u_n)} + \Phi_n(u_n)\leq
\Phi_n(T_{\Theta_n}^{(\lambda_n)}(u_n))$. Notice that for all
those $n\in\Natural$ such that
$\Phi_n'(T_{\Theta_n}^{(\lambda_n)}(u_n))\neq 0$, we have
\begin{align*}
\Phi_n(u_n) & \leq \Phi_n(T_{\Theta_n}^{(\lambda_n)}(u_n)) +
\innprod{u_n - T_{\Theta_n}^{(\lambda_n)}(u_n)}{\Phi'_n(u_n)} \\
& \leq \frac{\norm{\Phi_n'(T_{\Theta_n}^{(\lambda_n)}(u_n))}}{\nu_n}
\, \nu_n \frac{|\Phi_n(T_{\Theta_n}^{(\lambda_n)}(u_n))|}
     {\norm{\Phi_n'(T_{\Theta_n}^{(\lambda_n)}(u_n))}} + \norm{u_n -
       T_{\Theta_n}^{(\lambda_n)}(u_n)} \norm{\Phi'_n(u_n)} \\
& \leq \frac{D}{\epsilon'}
     \norm{(I-T_{\Phi_n}^{(\nu_n)})(T_{\Theta_n}^{(\lambda_n)}(u_n))}
     + D \norm{u_n - T_{\Theta_n}^{(\lambda_n)}(u_n)}.
\end{align*}

For all those $n\in\Natural$ where
$\Phi_n'(T_{\Theta_n}^{(\lambda_n)}(u_n)) = 0$, we have by
Definition~\ref{def:subgradient} that
$T_{\Theta_n}^{(\lambda_n)}(u_n)\in \argmin_{v\in\hilbert} \Phi_n(v)$,
and since $\lev\Phi_n \neq \emptyset$, we obtain
$\Phi_n(T_{\Theta_n}^{(\lambda_n)}(u_n))\leq 0$. Therefore, by similar
steps as previously, we obtain the following inequality for such
$n\in\Natural$: $\Phi_n(u_n) \leq D \norm{u_n -
  T_{\Theta_n}^{(\lambda_n)}(u_n)}$.

If we apply $\limsup_{n\rightarrow\infty}$ on both sides of the previous
inequalities, and if we recall
Theorems~\ref{thm:the.analysis}.\ref{thm:fraction.goes.to.zero} and
\ref{thm:the.analysis}.\ref{thm:T_n.demiclosed}, then we obtain 
$\limsup_{n\rightarrow\infty}\Phi_n(u_n) \leq 0$. Notice that in the
case where $\Phi_n :\hilbert \rightarrow [0,\infty)$, $\forall
  n\in\Natural$, then the previous analysis leads to
  $\lim_{n\rightarrow\infty} \Phi_n(u_n)=0$. This establishes
  Theorem~\ref{thm:subgrad.Thetan}.\ref{thm:Phin.go.to.zero}.

\item First, notice that since we work in the Euclidean space
  $\Real^L$, $\mathfrak{S}((\bm{u}_n)_{n\in\Natural}) =
  \mathfrak{W}((\bm{u}_n)_{n\in\Natural})$. Hence, the fact
  $\mathfrak{S}((\bm{u}_n)_{n\in\Natural}) \neq \emptyset$ is
  guaranteed by
  Theorem~\ref{thm:the.analysis}.\ref{thm:nonempty.weak.cluster.points}. Now,
  it can be verified that
  Theorem~\ref{thm:subgrad.Thetan}.\ref{thm:sparsity.example} is a
  direct consequence of
  Theorems~\ref{thm:the.analysis}.\ref{thm:nonempty.weak.cluster.points},
  \ref{thm:the.analysis}.\ref{thm:liminf.T_n}, and
  Lemma~\ref{lem:l1.function}.\ref{lem:l1.liminf.ass}.\qedhere

\end{enumerate}
\end{proof}


\section{Application: Online Sparsity-Aware System/Signal
  Recovery} \label{sec:application}

The present section will demonstrate the potential of the previously
introduced algorithms by devising a time-adaptive method for the
important, nowadays, sparse system/signal recovery task. In
particular, we will use Algorithm~\ref{ex:subgrad.Thetan} to derive a
low-complexity and similarly effective variant of the technique
introduced in \cite{Slav.Kops.Theo.ICASSP10, kst@tsp11}.

Sparsity is the key characteristic of systems or signals whose
representation, by means of some basis in some domain, consists of
only a few nonzero coefficients, while the majority of them retain
values of negligible size. The exploitation of sparsity has been
attracting recently an interest of exponential growth under the
\textit{Compressive Sensing} or \textit{Sampling (CS)} framework
\cite{Candes.ICM06, Donoho2006, Candes.Wakin.intro.cs}. In principle,
CS allows the estimation of sparse signals and systems using fewer
measurements than those previously thought to be necessary. More
importantly, recovery is realized by mobilizing efficient
constrained minimization schemes. Indeed, it has been shown that
sparsity is favored by $\ell_1$ constrained solutions
\cite{Candes.Wakin.intro.cs, CandesWakinBoyd08,
  Daubechies.sparse.projected.gradient.method, Daubechies.IRLS}.

Recall, here, that given two integers $j_1\leq j_2$, the notation
$\overline{j_1,j_2}$ stands for the set $\{j_1, j_1+1, \ldots,
j_2\}$. Assume a vector $\bm{x}_*:= [x_{*,1}, \ldots, x_{*,L}]^t$ in
the Euclidean space $\Real^L$, $L\in\Naturalstar$, where the
superscript $t$ stands for vector transposition. If the support of
$\bm{x}_*$ is defined as $\supp(\bm{x}_*):=
\{i\in\overline{1,L}:\ x_{*,i} \neq 0\}$, and the $\ell_0$ norm of
$\bm{x}_*$ is defined as the cardinality of its support, i.e.,
$\norm{\bm{x}_*}_{\ell_0}:= \# \supp(\bm{x}_*)$, by the term
``sparse'' $\bm{x}_*$, we refer to the case where
$\norm{\bm{x}_*}_{\ell_0}$ is considerably smaller than $L$.

The majority of CS techniques deal with the problem of estimating a
sparse system $\bm{x}_*$, based on a number $K (<L)$ of
measurements $(d_n)_{n=0}^{K-1} \subset\Real$ that are generated by
the following linear regression model (see \eqref{model}):
\begin{equation}
d_n = \bm{a}_n^t\bm{x}_*  + \zeta_n, \quad\forall n\in\Natural.
\label{linear.regression.model}
\end{equation}
Here, $(\bm{a}_n)_{n\in\Natural} \subset \Real^L$ are the input
vectors, which excite the unknown $\bm{x}_*$, and 
$(\zeta_n)_{n\in\Natural}$ is a real-valued discrete-time stochastic
process which stands for the contaminating additive noise.

A well-known \textit{batch} method for estimating the sparse
$\bm{x}_*$, based on a limited number $K<L$ of measurements, is
provided by the \textit{Least-Absolute Shrinkage and Selection
  Operator (LASSO) \cite{Tibshirani.Lasso, stanford.stats.book}:}
\begin{equation*}
\min \{\norm{\bm{A} \bm{x}-
  \bm{d}}^2:\ \norm{\bm{x}}_{\ell_1}\leq \norm{\bm{x}_*}_{\ell_1},
\bm{x}\in\Real^L\},
\end{equation*}
where $\norm{\cdot}$ stands for the classical Euclidean norm of a
vector, $\norm{\cdot}_{\ell_1}$ for the $\ell_1$ norm, i.e.,
$\norm{\bm{x}}_{\ell_1} := \sum_{j=1}^L |x_j|$, $\forall \bm{x}:=
[x_1, \ldots, x_L]^t\in \Real^L$, $\bm{d} := [d_0,\ldots, d_{K-1}]^t
\in\Real^K$, and $\bm{A}\in \Real^{K\times L}$ is the matrix whose
rows are $(\bm{a}_n^t)_{n=0}^{K-1}$. We stress here that the
term ``batch'' method means that the data $(\bm{a}_n,
d_n)_{n=0}^{K-1}$ have to be available prior to the application of LASSO.

With only a few recent exceptions, i.e., \cite{ChenHero09,
  angelosante.occd, myyy.soft.thres.icassp10,
  Slav.Kops.Theo.ICASSP10, kst@tsp11}, the majority of the proposed, so
far, CS techniques are appropriate for batch mode operation
\cite{Candes.Wakin.intro.cs, CandesWakinBoyd08, CandesRombergTao06,
  CandesTao05, Daubechies.sparse.projected.gradient.method,
  Daubechies.IRLS}. In other words, one has to wait until a fixed and
predefined number $K$ of training data $\{\bm{a}_n, d_n\}_{n=0}^{K-1}$
is available prior to application of CS processing methods, e.g.,
LASSO, in order to recover the corresponding signal/system
estimate. Dynamic online operation for updating and improving
estimates, as new measurements become available, is not feasible by
batch processing methods. The development of efficient,
\textit{time-adaptive, sparsity-aware} techniques is of great
importance in engineering, especially in cases where the signal or
system under consideration is time-varying and/or the available
storage resources are limited.

Moving along the path introduced in \cite{ChenHero09,
  angelosante.occd, myyy.soft.thres.icassp10, Slav.Kops.Theo.ICASSP10,
  kst@tsp11}, the present section will deal with the case where
$\bm{x}_*$ is not only sparse but it is also allowed to be
time-varying. For this reason, the number $K$ of available data is
allowed to take values towards $\infty$. In this sense, the studies
\cite{ChenHero09, angelosante.occd, myyy.soft.thres.icassp10,
  Slav.Kops.Theo.ICASSP10, kst@tsp11} operate in a framework that is
different than the standard CS scenario. The major objective is no
longer only the estimation of the sparse signal or system, based on a
limited number of measurements. Letting $K \rightarrow\infty$ in the
design, the additional task is the capability of the estimator to
track possible variations of the unknown sparse system. Moreover, this
has to take place at an affordable computational complexity, as
required by most real time applications, where time-adaptive
estimation is of interest. Consequently, the batch
\textit{sparsity-aware} techniques developed under the CS framework,
e.g., LASSO or one of its variants, become unsuitable under
time-varying scenarios.  The focus, now, becomes the development of a
framework that 1) exploits sparsity, 2) exhibits fast convergence to
error floors that are as close as possible to those obtained by their
batch counterparts, 3) offers good tracking performance, and 4) has
low computational demands in order to meet the stringent time
constraints that are imposed by most real time operation
scenarios. Such a framework was demonstrated in
\cite{myyy.soft.thres.icassp10, Slav.Kops.Theo.ICASSP10, kst@tsp11,
  skt@icassp11, kopsinis.dsp11, skt.eusipco11}. Here, we focus on
\cite{Slav.Kops.Theo.ICASSP10, kst@tsp11}. Motivated by the previously
presented Algorithm~\ref{ex:subgrad.Thetan}, we devise a variant of
\cite{Slav.Kops.Theo.ICASSP10, kst@tsp11}, which shows similar
performance to \cite{Slav.Kops.Theo.ICASSP10, kst@tsp11}, albeit its
lower computational requirements.

The information at our disposal is the sequence of training data
$(\bm{a}_n, d_n)_{n\in\Natural}$, the a-priori knowledge that the
unknown $\bm{x}_*$ in \eqref{linear.regression.model} is sparse, as
well as an estimate of the cardinality of the support of $\bm{x}_*$,
i.e., $\norm{\bm{x}_*}_{\ell_0}$. In the sequel, we will demonstrate a
way to incorporate the a-priori knowledge of the estimate of
$\norm{\bm{x}_*}_{\ell_0}$ in the design as a series of closed convex
sets.

In the spirit of Algorithm~\ref{ex:subgrad.Thetan}, we begin by
introducing a sequence of closed convex sets $(S_n)_{n\in\Natural}$,
which associate to the available training data $(\bm{a}_n,
d_n)_{n\in\Natural}$, and quantify the deviation from the adopted
model of \eqref{linear.regression.model} by the introduction of a
user-defined tolerance $\xi\geq 0$.

\begin{definition}[Closed hyperslab]\label{def:hyperslab}
Given the online training data $(\bm{a}_n, d_n)_{n\in\Natural} \subset
\Real^L \times \Real$, and a user-defined $\xi>0$, we define the
following sequence of closed convex sets, called \textit{closed
  hyperslabs:}
\begin{equation*}
\forall n\in\Natural, \quad S_n:= \{\bm{x}\in\Real^L: |d_n- \bm{a}_n^t
\bm{x}| \leq \xi\}.
\end{equation*}
The metric projection mapping $P_{S_n}$ can be analytically computed
\cite{tsy.sp.magazine, ksi.Beam.TSP}, it breaks down to the metric
projection onto a hyperplane, and its computational complexity scales
linearly to the number of unknowns $L$.
\end{definition}

In this section we mobilize Algorithm~\ref{ex:subgrad.Thetan}, where
$(S_n)_{n\in\Natural}$ becomes the sequence of closed hyperslabs of
Definition~\ref{def:hyperslab}, and $(\Phi_n)_{n\in\Natural}$ is the
sequence of sparsity-aware functions introduced in
Definition~\ref{def:l1.loss}. The Algorithm~\ref{ex:subgrad.Thetan},
with the metric projection mapping $P_{B_{\ell_1}[\bm{w}_n,\rho]}$
used instead of $T_{\Phi_n}^{(\nu_n)}$, was introduced in
\cite{Slav.Kops.Theo.ICASSP10, kst@tsp11}. The necessary complexity in
order to compute the $P_{B_{\ell_1}[\bm{w}_n,\rho]}$ is of order
$\mathcal{O}(L\log_2 L)$, needed for a sorting operation, and
$\mathcal{O}(L)$ multiplications and additions
\cite{Slav.Kops.Theo.ICASSP10, kst@tsp11}. In the present study, due
to the utilization of the relaxed subgradient projection mapping
$T_{\Phi_n}^{(\nu_n)}$ in Algorithm~\ref{ex:subgrad.Thetan}, together with
the simplicity of the subgradients of $\Phi_n$, seen in
Table~\ref{table:subgradient.linear.l1norm}, we are able to cut down
the computational complexity of the algorithm to $\mathcal{O}(L)$
operations. As it will be made clear by the subsequent numerical
experiments, the Algorithm~\ref{ex:subgrad.Thetan} results into a
similar performance to its predecessor \cite{Slav.Kops.Theo.ICASSP10,
  kst@tsp11}.

The reason for introducing a series of weighted $\ell_1$-balls
$B_{\ell_1}[\bm{w}_n, \rho]$, instead of the standard unweighted one
$B_{\ell_1}[\bm{1}, \rho]$, is that 1) we have observed that the
weighted $\ell_1$-balls, introduced in Definition~\ref{def:l1.loss},
offer enhanced convergence speed, as also demonstrated in
\cite{CandesWakinBoyd08, Daubechies.IRLS} in a different context, and
2) the weighted balls help us easily incorporate the a-priori
knowledge of the cardinality of the support of $\bm{x}_*$, i.e.,
$\norm{\bm{x}_*}_{\ell_0}$, in the radius $\rho$, as the following
lemma suggests.

\begin{lemma}\label{lem:l1.ball.radius} Assume that the sequence
  $(\bm{u}_n)_{n\in\Natural}$, generated by
  Algorithm~\ref{ex:subgrad.Thetan}, converges to the desirable
  $\bm{x}_*$. Then, there exists an $N\in \Natural_{\infty}$ such that
  $\forall \rho\geq 
  \norm{\bm{x}_*}_{\ell_0}$, $\forall n\in N$, $\bm{u}_n \in
  B_{\ell_1}[\bm{w}_n, \rho]$.
\end{lemma}

\begin{proof}
By definition, $\sum_{i=1}^L w_{n,i} |u_{n,i}| = \sum_{i=1}^L
\frac{|u_{n,i}|}{|u_{n,i}| + \check{\epsilon}}$, $\forall
n\in\Natural$. Since $\lim_{n\rightarrow\infty}\bm{u}_n =
\bm{x}_*$, 
\begin{align*}
\limsup_{n\rightarrow\infty} \sum_{i=1}^L w_{n,i} |u_{n,i}| & = 
\limsup_{n\rightarrow\infty} \sum_{i=1}^L \frac{|u_{n,i}|}{|u_{n,i}|
  + \check{\epsilon}} = \lim_{n\rightarrow\infty} \sum_{i=1}^L
\frac{|u_{n,i}|}{|u_{n,i}| + \check{\epsilon}} \\
& = \sum_{i\in\supp(\bm{x}_*)} \frac{|x_{*,i}|}{|x_{*,i}| + \check{\epsilon}} +
\sum_{i\notin\supp(\bm{x}_*)} \frac{|x_{*,i}|}{|x_{*,i}| + \check{\epsilon}}
< \sum_{i\in\supp(\bm{x}_*)} \frac{|x_{*,i}|}{|x_{*,i}|} =
\norm{\bm{x}_*}_{\ell_0}.
\end{align*}
The previous strict inequality and the definition of $\limsup$ suggest that
there exists an $N\in\Natural_{\infty}$ such that
$\forall n\in N$ we have $\sum_{i=1}^L w_{n,i} |u_{n,i}| \leq
\norm{\bm{x}_*}_{\ell_0}$. In other words, we obtain that $\forall
n\in N$, $\forall \rho\geq \norm{\bm{x}_*}_{\ell_0}$, $\bm{u}_n \in
B_{\ell_1}[\bm{w}_n, \rho]$. This establishes
Lemma~\ref{lem:l1.ball.radius}.
\end{proof}

In other words, Lemma~\ref{lem:l1.ball.radius} suggests that in order
to have the sequence $(\bm{u}_n)_{n\in\Natural}$ converge to
$\bm{x}_*$, a necessary condition is to set the radius $\rho$, in the
weighted balls $(B_{\ell_1}[\bm{w}_n, \rho])_{n\in\Natural}$, to a
value that over-estimates $\norm{\bm{x}_*}_{\ell_0}$. This strategy
will be followed in the subsequent numerical examples.

\subsection{Numerical examples}\label{sec:simulations}

In this section, the performance of the proposed algorithm is
evaluated for both time-invariant and time-varying systems. To save
space, only a couple of scenarios are considered. For extensive
experiments on the behavior of similar in spirit algorithms, the
interested reader is referred to \cite{kst@tsp11, kopsinis.dsp11}.

The proposed methodology is compared to a couple of recent
time-adaptive methods \cite{ChenHero09, angelosante.occd} which belong
to the same algorithmic family; the cost function to be minimized is
the sum of a quadratic loss, accounting for the regression model,
together with an $\ell_1$-norm regularization term, in order to infuse
sparsity into the design. The method RZ-LMS \cite{ChenHero09} is built
upon the classical Least Mean Squares (LMS) algorithm, and employs
re-weighting for the regularization term. Its computational complexity
scales linearly with respect to the system unknowns, i.e., it is of
order $\mathcal{O}(L)$. Re-weighting of the $\ell_1$-norm is also
utilized in OCCD-TNWL \cite{angelosante.occd}, where the quadratic
regression term follows the strategy in the celebrated Recursive Least
Squares (RLS) method, scoring an overall computational complexity of
order $\mathcal{O}(4L^2)$.

Moreover, we mobilized batch methods for solving the classical LASSO
\cite{Tibshirani.Lasso, BergFriedlander:2008, spgl1.2007}, as well as
its re-weighted variant \cite{Zou.lasso}. In other words, for every
batch method, each point in the respective curves is the outcome of a
sub-process which takes into account all the available data available
till the current time instant. It is clear that such an operation is
infeasible in real-time implementations. Nevertheless, these
performances will serve as benchmarks for the $\ell_1$-norm
regularized least squares solvers.

\begin{figure}
\includegraphics[scale=0.4]{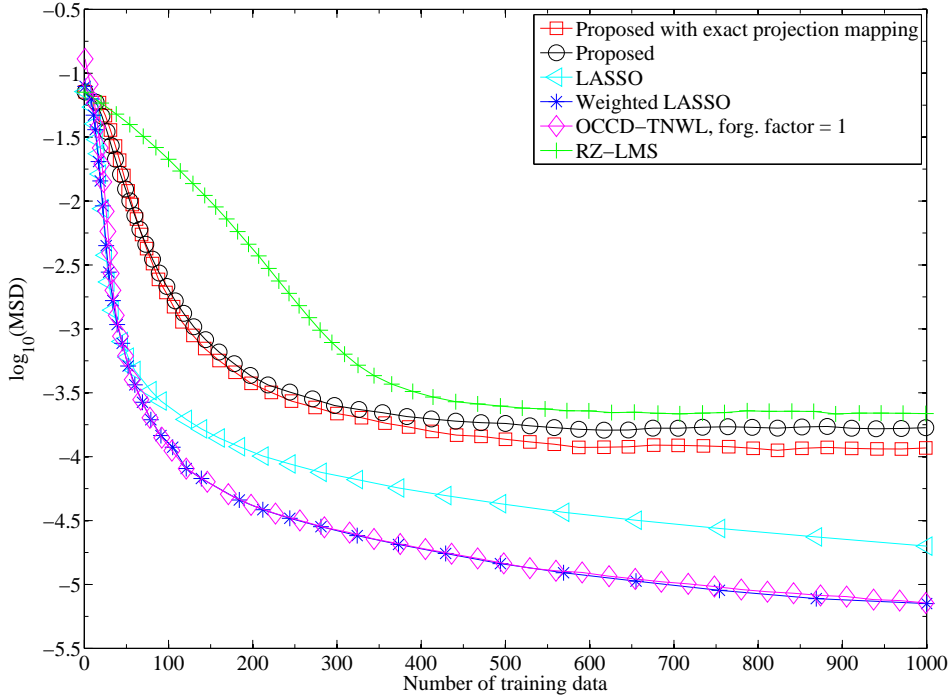}
\caption{Time-invariant sparse system $\bm{x}_*\in\Real^{100}$, with
  $\norm{\bm{x}_*}_{\ell_0}=5$. Here, the Mean Square Deviation (MSD)
  is defined as the following function on the number of the training
  data; $\text{MSD}(n) := \frac{1}{R} \sum_{r=1}^R \norm{\bm{x}_* -
    \bm{u}_n^{(r)}}^2$, $\forall n\in\Natural$, where $R$ is the total
  number of independent runs of the
  experiment. Here, $R:=300$.}\label{fig:time.invariant}
\end{figure}

Fig.~\ref{fig:time.invariant} refers to the case of a time-invariant
system $\bm{x}_*$, whose length is $L=100$ and only a number of $5$
coefficients, placed in arbitrary positions, are nonzero, i.e.,
$\norm{\bm{x}_*}_{\ell_0}=5$. The values of the nonzero coefficients
were drawn from a Gaussian distribution of zero mean and variance
equal to one. The input signal $(a_n)_{n\in\Integer}$ is defined as a
discrete-time Gaussian process of zero mean and variance equal to
$1$. The vectors $(\bm{a}_n)_{n\in\Natural}$, in
\eqref{linear.regression.model}, are formed as follows: $\forall
n\in\Natural$, $\bm{a}_n := [a_n, a_{n-1}, \ldots, a_{n-L+1}]^t$. The
noise process $(\zeta_n)_{n\in\Natural}$ is Gaussian with zero mean
and variance equal to $\sigma_n^2 := 0.1$.

In Fig.~\ref{fig:time.invariant}, the tag ``Proposed'' refers to
Algorithm~\ref{ex:subgrad.Thetan}. The curve ``Proposed with exact
projection mapping'' refers to Algorithm~\ref{ex:subgrad.Thetan}, but with
$P_{B_{\ell_1}[\bm{w}_n, \rho]}$ in the place of
$T_{\Phi_n}^{(\nu_n)}$, $\forall n\in\Natural$. This realization was
introduced in \cite{Slav.Kops.Theo.ICASSP10, kst@tsp11}. For both
``Proposed'' and ``Proposed with exact projection mapping'', $q$ was
set equal to $25$, $\omega_i^{(n)} := 1/\#\mathcal{I}_n$, $\forall
n\in\Natural$, in the cases where $\mathcal{I}_n\neq \emptyset$, $\rho:=6$,
$\check{\epsilon}:=0.005$, and $\xi := 2\sigma_n$.

All of the parameters for the methods ``LASSO''
\cite{Tibshirani.Lasso, BergFriedlander:2008, spgl1.2007}, ``Weighted
LASSO'' \cite{Zou.lasso}, ``OCCD-TNWL'' \cite{angelosante.occd}, and
``RZ-LMS'' \cite{ChenHero09} were tuned for producing the best
respective performance for the current setting. More specifically, the
forgetting factor for ``OCCD-TNWL'' \cite{angelosante.occd}, which is an
inherent parameter in any RLS-like scheme, was set equal to
$1$. Moreover, ``RZ-LMS'' \cite{ChenHero09} was tuned in such a
way for producing the lowest error floor for the iteration
$\#450$. Although different parameters for the ``RZ-LMS'' could result
into faster convergence speed, this could only be obtained at the
expense of higher error floors.

Fig.~\ref{fig:time.invariant} demonstrates that ``Proposed'' and
``Proposed with exact projection mapping'' lead to similar
performances. However, due to the mobilization of
$T_{\Phi_n}^{(\nu_n)}$ in ``Proposed'', the computational complexity
drops to $\mathcal{O}(qL)$, as opposed to $\mathcal{O}(qL + L\log L)$
in ``Proposed with exact projection mapping'', with $\mathcal{O}(L\log
L)$ accounting for sorting operations which are necessary for the
computation of the exact $P_{B_{\ell_1}[\bm{w}_n, \rho]}$. 

\begin{figure}
\includegraphics[scale=0.4]{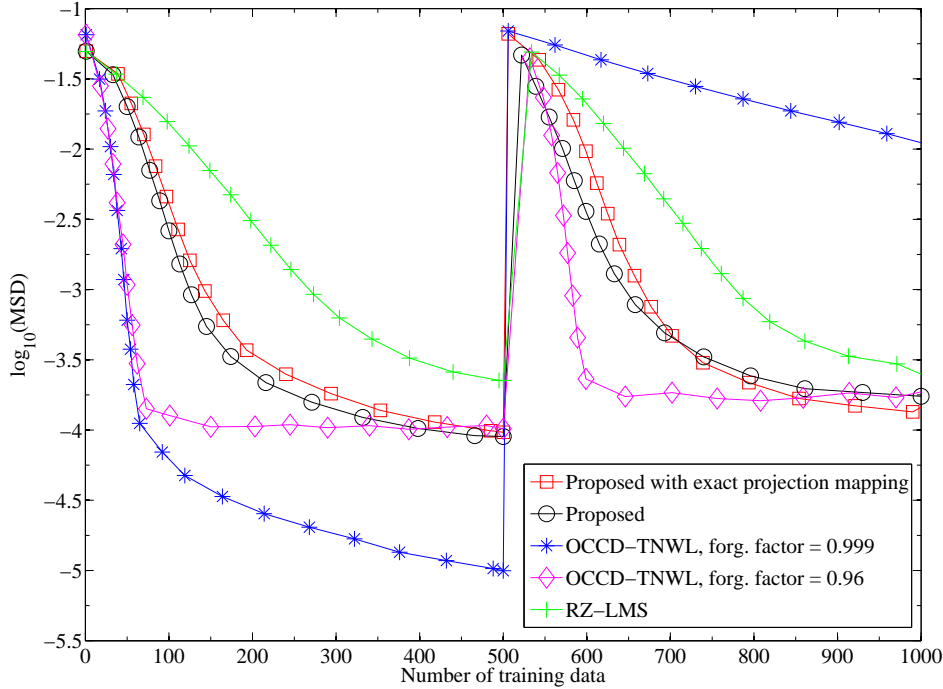}
\caption{Tracking performance for a time-varying sparse system
  $\bm{x}_*\in\Real^{100}$. The system $\bm{x}_*$ changes suddenly at
  the $\# 501$ time instant. Here, as in
  Fig.~\ref{fig:time.invariant}, $\text{MSD}(n) := \frac{1}{R}
  \sum_{r=1}^R \norm{\bm{x}_* - \bm{u}_n^{(r)}}^2$, $\forall
  n\in\Natural$, where $R$ is the total number of independent runs of
  the experiment. Similarly to Fig.~\ref{fig:time.invariant}, $R:=
  300$.}\label{fig:time.varying}
\end{figure}

Fig.~\ref{fig:time.varying} refers to the case of a time-varying
system. Both the number of nonzero elements of $\bm{x}_*$ and the
values of the system's coefficients are allowed to undergo sudden
changes. This is a typical scenario used in adaptive filtering in
order to study the tracking performance of an algorithm in
practice. The system used in the experiments is of dimension
$100$. The system change is realized as follows: For the first $500$
time instances, the first $5$ coefficients are set equal to $1$. Then,
at time instance $501$ the $\#2$ and $\#4$ coefficients are set equal
to zero, and all the odd coefficients from $\#7$ to $\#15$ are set
equal to $1$. Note that the sparsity level changes at time instance
$501$, and it becomes $8$ instead of $5$. The results are shown in
Fig.~\ref{fig:time.varying} with the noise variance being set equal to
$\sigma^2_n:= 0.1$.

Notice also here the similarity in the performance of ``Proposed'' and
``Proposed with exact projection mapping''. Moreover, the ``RZ-LMS''
shows better tracking ability than ``OCCD-TNWL'', with the forgetting factor
set equal to $0.999$. In order to raise the tracking ability of the
``OCCD-TNWL'', the method should be able to easily ``forget'' the remote
past and concentrate on recent variations of the system. This is
achieved by reducing the forgetting factor at the expense of an
increased error floor. We chose the value of $0.96$ for the forgetting
factor of the ``OCCD-TNWL'' in order to achieve similar error floor to the
``Proposed'' method, for both the employed sparse systems.

\section{Acknowledgments}

The authors would like to express their gratitude to the anonymous
reviewers, to Prof.~Sergios Theodoridis, Univ.~of Athens, Greece,
for the intriguing discussions and continuous support, to Dr.~Yannis
Kopsinis, for supplying us with lots of information on sparse
system/signal recovery and with the Matlab code for the experiments, and to
Assoc.~Prof.~Masahiro Yukawa, Niigata Univ., Japan, for his comments
which helped us to improve the original manuscript.


\end{document}